\documentclass[leqno,12pt]{amsart}

\pdfoutput=1

\usepackage[T1]{fontenc}
\usepackage[utf8]{inputenc}

\usepackage[ngerman,british]{babel}

\usepackage{lmodern}
\usepackage{microtype}

\usepackage[svgnames]{xcolor}
\usepackage[unicode=true,citecolor=ForestGreen,pdfusetitle]{hyperref}
\usepackage[ocgcolorlinks]{ocgx2} 

\let\oldmarginpar\marginpar
\def\marginpar#1{\oldmarginpar{\raggedright\linespread{1}
\tiny #1}}

\usepackage{amsmath,amssymb,amsthm}

\DeclareRobustCommand\de{\mathrm{d}}

\newtheorem{theorem}{\indent\sc Theorem}[section]
\newtheorem{lemma}[theorem]{\indent\sc Lemma}
\newtheorem{corollary}[theorem]{\indent\sc Corollary}
\newtheorem{claim}[theorem]{\indent\sc Claim}
\newtheorem{example}[theorem]{\indent\sc Example}

\theoremstyle{definition}

\numberwithin{equation}{section}

\usepackage{enumitem}

\thanks{No data are associated with this article. For the purpose of open access, the authors have
applied a Creative Commons Attribution (CC BY) licence to any Author Accepted Manuscript
version arising from this submission.}

\author{Vincenzo Mantova}
\address{School of Mathematics, University of Leeds, LS2 9JT, United Kingdom}
\email{V.L.Mantova@leeds.ac.uk}
\thanks{V.M. is supported by the EPSRC grant EP/T018461/1.}

\author{David Masser}
\address{\foreignlanguage{ngerman}{Departement Mathematik und Informatik, Universität Basel, Spiegelgasse 1, 4051 Basel}, Switzerland}
\email{David.Masser@unibas.ch}

\title{Polynomial-exponential equations -- some new cases of solvability}
\date{22nd July 2024}


\begin{document}

\begin{abstract}
  \noindent Recently Brownawell and the second author proved a ``non-degenerate'' case of the (unproved) ``Zilber Nullstellensatz'' in connexion with ``Strong Exponential Closure''. Here we treat some significant new cases. In particular these settle completely the problem of solving polynomial-exponential equations in two complex variables. The methods of proof are also new, as is the consequence, for example, that there are infinitely many complex $z$ with $e^z+e^{1/z}=1$.
\end{abstract}

\maketitle

\section{Introduction}\label{intro}

In a recent paper \cite{BM2017} Brownawell and the second author proved a result in connexion with Zilber's ``Strong Exponential Closure Axiom'' for ``pseudoexponential fields''. Over $\mathbf{C}$, this axiom becomes a conjecture, and it is formulated for irreducible algebraic varieties $\mathcal{V}$ in $\mathbf{C}^n \times \mathbf{C}^{*n}$. Its non-strong form, mentioned by Zilber \cite[Corollary 4.5, p.\ 83]{Zil2005} in connexion with $\mathbf{C}_{\rm exp}$ (see also Bays and Kirby \cite[pp.\ 495 and 538, 539]{BK2018} on ``exponential-algebraic closure''), states that if $\mathcal{V}$ is ``normal'' (or ex-normal) and ``free'' (see later for discussions of these concepts) then $\mathcal{V}$ has a point of the shape
\begin{equation}\label{SZ}
  (X_1,\ldots,X_n,\hat X_1,\ldots,\hat X_n)\ =\ (z_1,\ldots,z_n,e^{z_1},\ldots,e^{z_n}).
\end{equation}
This shape makes evident the connexion with Schanuel's Conjecture (see \cite{Lan1966} and Zilber \cite{Zil2002} for much more, as well as the work \cite{Man2016} of the first author and also the more recent papers \cite{DFT2018,DFT2021} of D'Aquino, Fornasiero and Terzo).

Proposition~2 of \cite[p.\ 448]{BM2017} proves the existence of \eqref{SZ} with these two hypotheses on $\mathcal{V}$ replaced by a single one. Denote by $\pi$ the projection from $\mathbf{C}^n \times \mathbf{C}^{*n}$ to $\mathbf{C}^n$. Then if (the Zariski closure) of $\pi(\mathcal{V})$ has dimension $n$, there is always such a point \eqref{SZ}. This single hypothesis implies that $\mathcal{V}$ is normal (but not that $\mathcal{V}$ is free). We remark that the Proposition as stated in \cite{BM2017} appears to require that $\mathcal{V}$ has dimension $n$; however if the dimension is $n'>n$ then life just gets simpler and we can adjoin $n'-n$ suitably chosen equations $\hat X_i=1$. For $n=1$ the Proposition seems to be reasonably well-known; see for example Marker \cite{Mar2006} (or Henson and Rubel \cite{HR1984} earlier). But already for $n=2$ it was new.

Our main purpose in this paper is to relax the condition that $\dim \pi(\mathcal{V}) = n$. We may note that if this dimension is 0, then $\pi(\mathcal{V})$ is a single point $(z_1,\ldots,z_n)$ and now the existence of \eqref{SZ} is obvious, provided $\mathcal{V}$ itself has dimension at least $n$.

Our main result concerns the case $\dim \pi(\mathcal{V})=1$. Here too it is reasonable to assume that $\mathcal{V}$ has dimension at least $n$: for example, it is not difficult to prove that if $\mathcal{L}$ is a line in $\mathbf{C}^n$ and $\mathcal{K}$ is a translate of a group subvariety of $\mathbf{G}_{\rm m}^n=\mathbf{C}^{*n}$ of dimension $n-2$, both generic in a perfectly explicit sense, then $\mathcal{L} \times \mathcal{K}$ contains no point \eqref{SZ}.

However there are more subtle obstructions to solvability. A simple example is $\mathcal{V}$ defined in $\mathbf{C}^2 \times \mathbf{C}^{*2}$ by
\begin{equation}\label{NZ}
  X_1+X_2=1, \quad \hat X_1\hat X_2=1.
\end{equation}
because if $z_1+z_2=1$ then $e^{z_1}e^{z_2}=e \neq 1$. There are similar examples with $X_1+X_2=1$ replaced by
\begin{equation}\label{LRS}
  m_1X_1+m_2X_2=c
\end{equation}
for any complex $c$ and integers $m_1,m_2$ not both zero. One can build analogous examples in $\mathbf{C}^n \times \mathbf{C}^{*n}$ where $\pi(\mathcal{V})$ is contained in the hyperplane defined by
\begin{equation}\label{AF}
  m_1X_1+\cdots +m_nX_n=c.
\end{equation}
It turns out that this is essentially the only obstruction under our new assumption. Thus we shall prove the following result.

\begin{theorem}\label{thm1}
  Suppose $\mathcal{V}$ is an irreducible algebraic variety in $\mathbf{C}^n \times \mathbf{C}^{*n}$, of dimension at least $n$, such that the Zariski closure of $\pi(\mathcal{V})$ in $\mathbf{C}^n$ has dimension 1. If \eqref{AF} does not hold on $\mathcal{V}$ with any complex $c$ and any integers $m_1,\ldots,m_n$ not all zero, then $\mathcal{V}$ contains a point \eqref{SZ}.
\end{theorem}

For $n=1$ this reduces again to the result in \cite{Mar2006}; but already for $n=2$ it is again new.

Note that the absence of relations \eqref{AF} is what Zilber calls ``free of additive dependencies (over $\mathbf{C}$)'' \cite[p.\ 74]{Zil2005}. When an obstruction \eqref{AF} does arise, one can recover further information inductively by means of ``back-substitution''.

For example in $\mathbf{C}^2 \times \mathbf{C}^{*2}$ with $\mathcal{V}$ defined by
\begin{equation}\label{EX}
  X_1+X_2=1, \quad X_1X_2=\hat X_1+\hat X_2
\end{equation}
any point \eqref{SZ} on $\mathcal{V}$ must lie on $\hat X_1\hat X_2 = e$, thus we can eliminate $X_2,\hat X_2$ to get down to
\begin{equation}\label{EXX}
  \hat X_1^2+e=X_1\hat X_1(1-X_1)
\end{equation}
in $\mathbf{C} \times \mathbf{C}^{*}$.

In Theorem~\ref{thm1} we focus on the existence of a single point \eqref{SZ}, but our proof gives rather more, as was the case in \cite{BM2017}. From the viewpoint of exponential polynomials it is natural to consider just the projections $(z_1,\ldots,z_n)$, although we will mention another possibility later. In \cite{BM2017} we remarked that the set of all such projections is not only infinite but even Zariski dense in $\mathbf{C}^n$ (which is $\pi(\mathcal{V})$ there). This was proved by a simple trick (see later); unfortunately that does not work in our situation if $n>1$. Nevertheless our method of proof shows that the set of $(z_1,\ldots,z_n)$ is indeed infinite and therefore Zariski dense in the curve $\pi(\mathcal{V})$.

In fact for $n=1$ the known results already lead to a fairly explicit description of the set $Z=Z_\mathcal{C}$ of points \eqref{SZ} on an irreducible curve $\mathcal{C}$ in $\mathbf{C} \times \mathbf{C}^*$ based only on the geometry of $\mathcal{C}$. Namely, if $\dim \pi(\mathcal{C})=0$, then $\pi(Z)$ is a single point, while if $\dim \pi(\mathcal{C})=1$, then $\pi(Z)$ is infinite, hence Zariski dense in $\pi(\mathcal{C})$.

Now for $n=2$ we can combine Theorem~\ref{thm1} with back-substitution and \cite{BM2017} to give a conclusive result for $\pi(Z)$ and surfaces $\mathcal{S}$ in $\mathbf{C}^2 \times \mathbf{C}^{*2}$ (we remind the reader that Schanuel's Conjecture itself remains unknown for $n=2$). To state this the following notation will be useful.

If the Zariski closure of $\pi(\mathcal{S})$ is a line $\mathcal{L}$ as in \eqref{LRS}, which we may call a line with rational slope, we write $\mathcal{K}$ for the set of $(e^{z_1},e^{z_2})$ in $\mathbf{C}^{*2}$ with $(z_1,z_2)$ in $\mathcal{L}$. This is an algebraic curve (and even a translate of a group subvariety). Then we write $\mathcal{G = L \times K}$ (still a translate of a group subvariety) and
\begin{equation}\label{GT}
  \mathcal{T  = S \cap G}.
\end{equation}
We will see in Section~\ref{subcases} that if the variety $\mathcal{T}$ is non-empty then it is infinite.

Our result for $n=2$ is as follows.

\begin{theorem}\label{thm2}
  Suppose $\mathcal{S}$ is an irreducible surface in $\mathbf{C}^2 \times \mathbf{C}^{*2}$, and let $Z=Z_\mathcal{S}$ be the set of points \eqref{SZ} in $\mathcal{S}$.
  \begin{enumerate}[label=\textup{(}\alph*\textup{)}]
    \item\label{dim-pi-S-0} If $\dim \pi(\mathcal{S})=0$ then $\pi(Z)$ is a single point, so Zariski dense in $\pi(\mathcal{S})$.
    \item\label{dim-pi-S-2} If $\dim \pi(\mathcal{S})=2$ then $\pi(Z)$ is infinite, and even Zariski dense in $\pi(\mathcal{S})$.
    \item\label{dim-pi-S-1-free} If $\dim \pi(\mathcal{S})=1$ and the Zariski closure of $\pi(\mathcal{S})$ is not a line of rational slope, then $\pi(Z)$ is infinite, so Zariski dense in $\pi(\mathcal{S})$.
    \item\label{dim-pi-S-1-nonfree} If $\dim \pi(\mathcal{S})=1$ and the Zariski closure of $\pi(\mathcal{S})$ is a line of rational slope $\mathcal{L}$, then the following subcases depending on $\mathcal{T}$ in~\eqref{GT} are possible:
          \begin{enumerate}[label=\textup{(}d\textsubscript{\textup{\arabic*}}\textup{)}]
            \item\label{dim-T--1} If $\mathcal{T}$ is empty then $\pi(Z)$ is empty.
            \item\label{dim-T-2} If $\dim \mathcal{T}=2$ then $\pi(Z)$ is infinite, so Zariski dense in $\pi(\mathcal{S})$.
            \item\label{dim-T-1} If $\dim \mathcal{T}=1$ then the following subcases are possible:
                  \begin{enumerate}[label=\textup{(}d\textsubscript{\textup{3\arabic*}}\textup{)}]
                    \item\label{dim-pi-T-0} If $\dim \pi(\mathcal{T})=0$ then $\pi(Z)$ is non-empty and finite, so not Zariski dense in $\pi(\mathcal{S})$.
                    \item\label{dim-pi-T-1} If $\dim \pi(\mathcal{T})=1$ then $\pi(Z)$ is infinite, so Zariski dense in $\pi(\mathcal{S})$.
                  \end{enumerate}
          \end{enumerate}
  \end{enumerate}

\end{theorem}
In particular the only situation where there are no points \eqref{SZ} in $\mathcal{S}$ is \ref{dim-T--1} of \ref{dim-pi-S-1-nonfree}; a typical example is \eqref{NZ}, where now $\mathcal{L}$ is defined by $X_1 + X_2 = e$ and $\mathcal{K}$ by $\hat X_1\hat X_2=e$. We stress that all possible cases and subcases may happen.

\subsection*{A proof sketch} The key ideas of our proofs build on those of \cite{Mar2006} for $n=1$. We proceed to recall the arguments there for the example $X_1=\hat X_1$.

To solve the resulting $e^z=z$ we look at the function $\Phi(z)=e^z-z$, an entire function of order at most 1. Using Hadamard's Factorization Theorem as in \cite{Mar2006} or better \cite[XIII\ 3.5]{Lan1999} we find that if $\Phi$ has no zeroes then it is $e^{\phi}$ for $\phi$ entire. A standard application of Borel-Carathéodory (see below) shows that $\phi$ is a polynomial of degree at most 1. So there would be $a,b$ in $\mathbf{C}$ with
\begin{equation}\label{ab}
  e^z-z=e^{az+b}.
\end{equation}
This can be disproved in an elementary way by repeated differentiation, or less elementary using algebraic structure theorems of van den Dries \cite{Dri1984} and of Henson and Rubel \cite{HR1984}; in our situation for general $n$ it suffices to apply a well-known result of Ax \cite{Ax1971} (which in fact is a functional analogue of Schanuel's Conjecture).

Let us examine more closely why our results are new for $n=2$. Consider the example
\[ X_1X_2=1, \quad \hat X_1+\hat X_2=1. \]
Solving for \eqref{SZ} is equivalent to solving the single (non-polynomial-exponential) equation
\begin{equation}\label{a}
  e^z+e^{1 / z}=1
\end{equation}
in complex numbers $z \neq 0$, which does not seem trivial but one sees no obvious obstruction. The basic argument in \cite{BM2017}, finding a fairly obvious approximate solution (like $z_0=200\pi i$ to $e^z=z$), then refining it and then using Newton's Method to home in on an actual solution, seems numerically to give convergence (see also Section~\ref{effectivity}). However for
\begin{equation}\label{b}e^z+e^{z^2}=1
\end{equation}
and $z_0=(200\pi i)^{1/2}=\sqrt{100\pi}(1+i)$ the matter is less clear, and numerically there are hints of the ``chaos'' which is well-known to exist in Newton's Method, with convincing convergence only after 20 iterations. Curiously enough, it works fine for $z_0=-\sqrt{100\pi}(1+i)$.

In fact the method used in \cite{Mar2006} works very well for \eqref{b}: we find instead of \eqref{ab}
\[ e^z+e^{z^2}-1=e^{az^2+bz+c} \]
which can be disproved as before. Our main contribution in this paper is to show that it extends to \eqref{a} and our general situation.

As it stands, this argument fails for \eqref{a}, because we have an essential singularity at $z=0$. Thus we have to restrict to $\mathbf{C}\setminus \{0\}$. But generally $\Phi$ analytic on $\mathbf{C}\setminus \{0\}$ with no zeroes on $\mathbf{C}\setminus \{0\}$ need not be $e^\phi$ for $\phi$ analytic on $\mathbf{C}\setminus \{0\}$. A simple counterexample is $\Phi(z)=z$.

Littlewood \cite[p.\ 392]{GMS1963} said ``it can pay to find out what is the worst enemy of what you want to prove, and then induce him to change sides''. This we do here; the argument to prove $\Phi=e^\phi$ constructs $\phi$ as
\[ \int{\frac{\Phi'(z)}{\Phi(z)}}\de z, \]
and all we have to do is stay inside $\mathbf{C}\setminus \{0\}$ and ensure that the integral along the homology loops containing $z=0$ vanishes, which can be arranged by multiplying $\Phi$ by a power of the ``enemy'' $z$. The upshot for \eqref{a} is that
\begin{equation}\label{HFT}e^z+e^{1/z}-1=z^me^{\phi(z)}
\end{equation}
for some integer $m$ and some $\phi$ analytic on $\mathbf{C}\setminus \{0\}$.

But now we will have to be more careful with Borel-Carathéodory and $z=0$, and in fact it yields only
\begin{equation}\label{lau}
  \phi(z)=az^2+bz+c+{d \over z}+{e \over z^2}
\end{equation}
now with an exponent of $z$ bigger than one might expect and of course of $1/z$ too. Nevertheless still Ax-Lindemann leads to a contradiction.

For examples like
\[ e^z+e^{1/(z^3+z+1)}=1 \]
we have to avoid three points, so three enemies, namely the three factors of $z^3+z+1$, and leading to homology of rank 3.

More generally, we have

\begin{example}\label{ex1}
  Suppose $\Phi$ is analytic on $\mathbf{C} \setminus \{p_1,\ldots,p_{s-1}\}$ and never vanishes there. Then there are integers $m_1,\ldots,m_{s-1}$ such that
  \[ \Phi(z)=(z-p_1)^{m_1}\cdots(z-p_{s-1})^{m_{s-1}}e^{\phi(z)} \]
  for some $\phi$ also analytic on $\mathbf{C} \setminus \{p_1,\ldots,p_{s-1}\}$.
\end{example}
Consider next the surface
\[ X_1^3+X_1+1=X_2^2, \quad \hat X_1+\hat X_2=1 \]
leading to
\[ e^z+e^{\sqrt{z^3+z+1}}=1. \]
As the projection to $\mathbf{C}^2$ is the affine part of an elliptic curve $\mathcal{E}$, we can no longer work with $\mathbf{C}\setminus S$ for a finite set $S$, and this particular problem concerns $\mathcal{E}\setminus \{O\}$ for the origin $O$. In fact the homology is the same as that of $\mathcal{E}$, with rank 2, and so we have to find two enemies. These can be written down explicitly on the universal cover $\mathbf{C}$ of $\mathcal{E}$ in terms of Weierstrass $\wp$ and $\zeta$ functions, or by integrating suitable differentials of the first and second kind on $\mathcal{E}\setminus \{O\}$. They are in fact the very simplest examples of Baker-Akhiezer functions (see~\cite[Ch.\ XIV]{Bak1995} and the foreword by Krichever), although they were known to Weierstrass. For these there seem to be no algebraic structure theorems, but again Ax-Lindemann suffices for a contradiction.

To see this in action, represent $\mathcal{E} \setminus \{O\}$ in the form $y^2 = 4x^3 - g_2 x - g_3$ (a simple change of variables suffices). Recall that $\mathcal{E}$ can be seen as the quotient of $\mathbf{C}$ by a two-dimensional lattice $\Omega$, so we may think of functions on $\mathcal{E} \setminus \{O\}$ as doubly periodic functions on $\mathbf{C} \setminus \Omega$. The associated Weierstrass $\zeta$ has the property that for every $\omega \in \Omega$, there is a quasiperiod $\eta$ such that $\zeta(z + \omega) = \zeta(z) + \eta$. Explicit enemies are then the functions $e^{\omega\zeta(z) - \eta z}$ for any non-zero $\omega$. The outcome is the following analogue of \eqref{HFT}.

\begin{example}\label{ex2}
  Suppose $\Phi$ is doubly periodic with respect to $\Omega$, analytic on $\mathbf{C} \setminus \Omega$ and never vanishes there. Then there is a period $\omega$, with quasi-period $\eta$, such that
  \[ \Phi(z)=e^{\omega\zeta(z)-\eta z}e^{\phi(z)} \]
  for some $\phi$ also doubly periodic with respect to $\Omega$ and analytic on $\mathbf{C} \setminus \Omega$.
\end{example}

Going further, to solve
\[ e^z+e^{1/\sqrt{z^3+z+1}}=1 \]
we have to avoid an additional three points (the zeroes of $z^3 + z + 1$), leading to homology of rank 5 and differentials of the third kind or the Weierstrass sigma function (also Baker-Akhiezer). Here too we get an analogue of \eqref{HFT}. See \eqref{ex3},~\eqref{ex4},~\eqref{ex5} in Section~\ref{examples} for more examples.

Finally consider
\begin{equation}\label{fermat}
  X_1^9+X_2^9=1, \quad \hat X_1+\hat X_2=1
\end{equation}
leading to a curve $\mathcal{C}$ of genus 28 and homology rank 56. To write down the enemies as complex functions (on the covering space) is not so easy without the aid of theta functions (in 28 variables), but on the curve the 56 enemies correspond again to a suitable choice of differentials of the first and second kind (and generally we need the third kind too).

After working out our proofs in terms of these explicitly constructed enemies we realized that there is a more abstract proof based on the canonical isomorphism between algebraic and analytic de~Rham cohomology of complex affine varieties, dating back to Grothendieck \cite{Gro1966}. That paper actually uses Hironaka's resolution of singularities, but we found that this was not needed in our situation (see Section~\ref{periods} for more details). So in the end we were able to avoid any appeal to \cite{Gro1966} by using instead suitable differentials on the underlying curve. This is the proof that we present here; nevertheless we do give an account of the original more explicit constructions, also because these seem to be helpful in obtaining effective versions of our results in which, for example, the zeroes can be localized.

\subsection*{Further remarks} In \cite{BM2017}, with the points \eqref{SZ} on $\mathcal{V}$ projecting to a Zariski dense subset of $\mathbf{C}^n$, we noted that this holds even in a strong sense of being ``relatively near'' to any one of ``sufficiently many'' points on $(2\pi i\mathbf{Z})^n$. It would be interesting to obtain similar strengthenings in our present set-up.

It is also natural to consider the distribution of the unprojected points \eqref{SZ}. In the situation of \cite{BM2017} the trick extends at once to show that they are Zariski dense in $\mathcal{V}$ itself. For if $G$ in $\mathbf{C}[X_1,\ldots,X_n,\hat X_1,\ldots,\hat X_n]$ does not vanish on $\mathcal{V}$, we may apply \cite{BM2017} to the variety in $\mathbf{C}^{n+1} \times \mathbf{C}^{*n+1}$ defined by the equations of $\mathcal{V}$ together with $G=\hat X_{n+1}$.

In our situation the analogous statement is unclear, even for $n=2$. This is illustrated by case \ref{dim-pi-S-1-free} in Theorem~\ref{thm2}. Just for the example \eqref{fermat} the density in $\mathcal{S}$ would amount to the fact that there is no $G \neq 0$ in $\mathbf{C}[X_1,\hat X_1]$ such that $G(z,e^z)=0$ for all $z$ with $e^z+e^{\root 9 \of {1-z^9}}=1$, which does not seem obvious.

In fact this case \ref{dim-pi-S-1-free} is the only problem, as will be established during the proof.

It would also be interesting to extend the investigations to $\pi(\mathcal{V})$ of other dimensions. The simplest case of dimension 2 in $\mathbf{C}^3 \times \mathbf{C}^{*3}$ leads to systems of equations such as
\[ e^z+e^{z^2-w^2}=z, \quad e^w+e^{z^2-w^2}=-w. \]

But it may be more difficult to find corresponding extensions of Theorem~\ref{thm2}. This is because of the obstructions coming from the concept of ``normal'' (see \cite[p.\ 75]{Zil2005}). For $\mathbf{C}^n \times \mathbf{C}^{*n}$ it amounts to the following. For $k=1,\ldots,n$ and a matrix of $k$ independent rows and $n$ columns with integer entries $m_{ij}$ we define a map $\mu$ from $\mathbf{C}^n \times \mathbf{C}^{*n}$ to $\mathbf{C}^k \times \mathbf{C}^{*k}$ by
\[ \mu(X_1,\ldots,X_n,\hat X_1,\ldots,\hat X_n)=\left(\sum_{j=1}^nm_{1j}X_j,\ldots,\sum_{j=1}^nm_{kj}X_j,\prod_{j=1}^n\hat X_j^{m_{1j}},\ldots,\prod_{j=1}^n\hat X_j^{m_{kj}}\right). \]
Then one imposes the condition that $\dim\mu(\mathcal{V}) \geq k$ for all $k$ and $\mu$.

If this condition fails then some $\mu(\mathcal{V})$ may be too small to contain the analogues of points \eqref{SZ}.

Note that normality is not a necessary condition; when $n=2$ it fails for the example
\[ X_1+X_2=1, \qquad \hat X_1\hat X_2=e \]
with $k=1$ and $m_{11}=m_{12}=1$ whereas here $Z$ is clearly infinite. We note that this surface is not ``free'' (see \cite[pp.\ 74--75]{Zil2005}); in fact, it turns out that for $n=2$, ``freeness'' and ``normality'' are equivalent.

We briefly mention some aspects of decidability and effectivity in our Theorem~\ref{thm1}.

Already in Theorem~\ref{thm2} for $n=2$ it may not be possible to decide for a given $\mathcal{S}$ which of the various possibilities actually arises. Say the defining equations are just $X_1=1$, $\hat{X}_1=\theta$ for some $\theta$ in ${\bf C}$, so we are in \ref{dim-pi-S-1-nonfree}; we find that we are in \ref{dim-T--1} or \ref{dim-T-2} according to whether $e \neq \theta$ or $e = \theta$. Now if $\theta$ is an explicitly given element of $\overline{{\bf Q}(\pi)}$, this may not be so easy, as for example with
\[ \sqrt[6]{\pi^5+\pi^4} = 2.7182818086\ldots; \]
and in general it involves Schanuel's Conjecture of course (see \cite[p.\ 31]{Lan1966}).

Furthermore for three-folds in $\mathbf{C}^3 \times \mathbf{C}^{*3}$ another obstacle arises. Take any absolutely irreducible polynomial $P$ in two variables over $\mathbf{Q}$, and $\mathcal{V}$ defined by $P(X_1/(2\pi i), X_2/(2\pi i)) = 0$ together with $\hat{X}_1=1$, $\hat{X}_2=1$. Then the points \eqref{SZ} correspond exactly to the integral solutions of $P(x_1, x_2) = 0$. If the genus here is at least 2 then we know no algorithm for finding these, as for example with
\[ x_1^4-2x_2^4+x_1x_2+x_1-n=0, \]
and in general it involves Hilbert's Tenth Problem (see for example \cite{Dav1973}).

We note also that just to decide if a variety is normal involves Zilber-Pink matters (specifically the conjectures of intersection with tori or cosets as considered in \cite{Zil2002,Zil2005} for example).

Nevertheless we shall give some simple effectivity arguments for special cases of Theorem~\ref{thm1} such as \eqref{a}.

\subsection*{Structure of the paper} The rest of our paper is arranged as follows.

In Section~\ref{prelim} we record some preliminary observations towards the proof of Theorem~\ref{thm1}, including the ``punctured'' version of Borel-Carathéodory which avoids $z=0$ and also the version of Ax's Theorem that we need.

Then in Section~\ref{periods} we construct suitable differentials on the underlying curve $\pi(\mathcal{V})$ and deduce our special form of the Grothendieck result.

The proof of Theorem~\ref{thm1} follows in Section~\ref{main-proof}, and that of Theorem~\ref{thm2} in Section~\ref{subcases}.

Then in Section~\ref{examples} we present the explicit versions of \eqref{HFT} for small genus, and finally in Section~\ref{effectivity} we briefly explain about effectivity and Theorem~\ref{thm1}.

We are grateful to David Grant for his help in connexion with the genus 2 constructions in Section~\ref{examples}.

\section{Preliminaries}\label{prelim}
From now on, $\mathcal{V}$ will be as in Theorem~\ref{thm1}. When $\dim \mathcal{V}=n$ it will be important to know that $\mathcal{V}$ is defined, apart from the equations defining the curve $\pi(\mathcal{V})$, by a single additional equation.

\begin{lemma}\label{lem1}
  Suppose $\mathcal{V}$ is an irreducible algebraic variety in $\mathbf{C}^n \times \mathbf{C}^{*n}$ of dimension $n$ such that the Zariski closure of $\pi(\mathcal{V})$ in $\mathbf{C}^n$ is a curve $\mathcal{C}_0$. Let $\mathfrak{P}_0$ be the prime ideal of $\mathcal{C}_0$ in $\mathfrak{R}_0=\mathbf{C}[X_1,\ldots,X_n]$. Similarly, write $\mathfrak{P}$ for the prime ideal of $\mathcal{V}$ in $\mathfrak{R} = \mathbf{C}[X_1,\ldots,X_n,\hat X_1,\ldots,\hat X_n]$.

  Then there are $F$ in $\mathfrak{P}$, not in $M\mathfrak{R}_0 + \mathfrak{P}_0 \mathfrak{R}$ for any monomial $M$ in $\hat X_1,\ldots,\hat X_n$, and $G_0$ in $\mathfrak{R}_0$, not in $\mathfrak{P}_0$, such that $\mathfrak{P}$ is contained in $G_0^{-1}(F\mathfrak{R} + \mathfrak{P}_0\mathfrak{R})$. Furthermore if $\mathcal{C}_0$ is a line then we can take $G_0=1$; in particular, if $n=2$ and $\mathcal{C}_0$ is defined by the vanishing of a polynomial $F_0$ of degree $1$, then $\mathfrak{P} = F\mathfrak{R} + F_0\mathfrak{R}$.
\end{lemma}

\begin{proof}
  Let $x_1,\ldots,x_n$ denote the coordinate functions on $\mathcal{C}_0$, and call $\mathfrak{r}_0$ its coordinate ring $\mathbf{C}[x_1,\ldots,x_n]$. By considering elements of $\mathfrak{R}$ as polynomials in  $\hat X_1,\dots,\hat X_n$ we see that the specialization $\sigma$ from $\mathfrak{R}$ to $\mathfrak{r}_0[\hat X_1,\ldots,\hat X_n]$ has kernel $\mathfrak{P}_0 \mathfrak{R}$ (which lies in $\mathfrak{P}$). We now pass into $K_0[\hat X_1,\ldots, \hat X_n]$ with the quotient field $K_0=\mathbf{C}(x_1,\ldots,x_n)$ of $\mathfrak{r}_0$. We claim that $\sigma(\mathfrak{P})$ is prime in $\mathfrak{r}_0[\hat X_1,\ldots,\hat X_n]$, and then by clearing denominators we see that $K_0 \sigma(\mathfrak{P})$ is prime in $K_0[\hat X_1,\dots,\hat X_n]$ too.

  Indeed, suppose $P_1$, $P_2$ are in $\mathfrak{r}_0[\hat X_1,\ldots,\hat X_n]$ with $P_1P_2$ in $\sigma(\mathfrak{P})$. Clearly $P_1=\sigma(Q_1)$, $P_2=\sigma(Q_2)$ with $Q_1$, $Q_2$ in $\mathfrak{R}$, and $P_1P_2=\sigma(Q)$ for $Q$ in $\mathfrak{P}$. Thus $Q_1Q_2-Q$ is in the kernel $\mathfrak{P}_0 \mathfrak{R}$ so in $\mathfrak{P}$. So also $Q_1Q_2$ is in $\mathfrak{P}$. If $Q_1$ is in $\mathfrak{P}$ then $P_1$ is in $\sigma(\mathfrak{P})$ and similarly for $P_2$; this gives the above claim.

  Now $\mathfrak{R}/\mathfrak{P}$ has transcendence degree $n$ over $\bf C$, and it has a subfield $\mathfrak{R}_0/\mathfrak{P}_0$ of transcendence degree $1$ over $\bf C$; thus $\mathfrak{R}/\mathfrak{P}$ has transcendence degree $n-1$ over $\mathfrak{R}_0/\mathfrak{P}_0$. It follows (see for example \cite[p.\ 91]{ZS1960}) that $K_0 \sigma(\mathfrak{P})$ is minimal in the sense of \cite[p.\ 238]{ZS1958}. This latter reference (``Principal Ideal Theorem'') shows that $K_0 \sigma(\mathfrak{P})$ is principal, as $K_0[\hat X_1,\ldots,\hat X_n]$ is a unique factorization domain. We can further assume the generator is in $\sigma(\mathfrak{P})$, so it is $\sigma(F)$ for some $F$ in $\mathfrak{P}$. If $F$ were in $M\mathfrak{R}_0 + \mathfrak{P}_0 \mathfrak{R}$ for some monomial $M$ as above then $\sigma(F)=\sigma(M)f$ for some non-zero $f$ in $\mathfrak{r}_0$, so we could have taken the generator as $M$; however that would imply that the projection of $\mathcal{V}$ to $\mathbf{C}^{*n}$ is empty, which is certainly not the case.

  Finally for each $A$ in $\mathfrak{P}$ there is $h_A$ in $K_0[\hat X_1,\ldots,\hat X_n]$ with $\sigma(A)=h_A\sigma(F)$ and there is $g_A \neq 0$ in $\mathfrak{r}_0$ with $g_Ah_A$ in $\mathfrak{r}_0[\hat X_1,\ldots,\hat X_n]$. In particular $g_A=\sigma(G_A)$ for some $G_A$ in $\mathfrak{R}_0$, and $g_Ah_A= \sigma(B_A)$ for some $B_A$ in $\mathfrak{R}$. Hence $A$ is in $G_A^{-1}(F\mathfrak{R} + \mathfrak{P}_0\mathfrak{R})$ and the result follows on taking a finite basis for $\mathfrak{P}$.

  If $\mathcal{C}_0$ is a line then $\mathfrak{r}_0$ is isomorphic to some $\mathbf{C}[x]$, thus a unique factorization domain. Then $\mathbf{C}[x][\hat X_1,\ldots,\hat X_n]$ is a unique factorization domain, so $\sigma(\mathfrak{P})$ is principal and we may assume $\sigma(F)$ to be its generator, in which case we find $\mathfrak{P} = F\mathfrak{R} + \mathfrak{P}_0\mathfrak{R}$, or in other words, we may always take $g_A = 1$. In the special case $n=2$ with $\mathfrak{P}_0$ generated by a polynomial $F_0$ of degree $1$, we find $\mathfrak{P} = F\mathfrak{R} + F_0\mathfrak{R}$.

  This completes the proof.
\end{proof}
Thus to find a point \eqref{SZ} on $\mathcal{V}$ it suffices to solve $F(z_1,\ldots,z_n,e^{z_1},\ldots,e^{z_n})=0$ with $(z_1,\ldots,z_n)$ on $\mathcal{C}_0$ but $G_0(z_1,\ldots,z_n) \neq 0$.

Next we recall a standard version of Borel-Carathéodory which estimates the absolute value $|\phi_0|$ in terms of the real part $\Re \phi_0$.

\begin{lemma}\label{lem2}
  For $0 \leq r<R$ and any $\phi_0$ analytic on the disc $|w| \leq R$ we have
  \[ \sup_{|w|\leq r}|\phi_0(w)|\ \leq \ {2r \over R-r}\sup_{|w|\leq R}\Re \phi_0(w) +{R+r \over R-r}|\phi_0(0)|. \]
\end{lemma}

\begin{proof} See for example \cite[XII 3.1]{Lan1999}.
\end{proof}
\bigskip
This is used in the classical theory to deduce from an inequality $\Re \phi_0(w) \leq c|w|^\kappa\ (\kappa>0)$, for $\phi_0$ entire and all $|w|$ large, a similar inequality $|\phi_0(w)| \leq c'|w|^\kappa$. We use it here to obtain the following consequence for functions $\phi$ analytic only near (but not at) a finite point, which we can take as $z=0$.

\begin{lemma}\label{lem3}
  Let $\phi$ be a function analytic on a punctured neighbourhood of $0$ on which
  \[ \Re\phi(z) \leq \frac{c}{|z|^\kappa} \]
  for some real $c,\kappa\geq0$ independent of $z$. Then there is a punctured neighbourhood of $0$ on which
  \[ |\phi(z)| \leq \frac{c'}{|z|^{\kappa+1}} \]
  for some real $c'$ independent of $z$.
\end{lemma}
\begin{proof}
  Let $0 < |z| \leq 2\delta$ be a punctured neighbourhood as in the assumption. The conclusion will be about the neighbourhood $0 < |z| \leq \delta$. Write $c_1=\sup_{|z|=\delta}|\phi(z)|$, and choose any $z_1$ with $0 < |z_1| \leq \delta$. We are going to apply Lemma~\ref{lem2} to carefully chosen $w$-discs with
  \[ w = z - \delta \frac{z_1}{|z_1|}, \]
  so that their centre $w = 0$ lies on $|z| = \delta$. Accordingly define
  \[ \phi_0(w) = \phi\left(w + \delta\frac{z_1}{|z_1|}\right) \]
    and
  \[ r = \delta - |z_1|, \quad R = \delta - \frac{|z_1|}{2}, \]
  so that $0 \leq r < R$. Now $|w| \leq R$ implies
  \[ |z| = \left|w + \delta\frac{z_1}{|z_1|}\right| \leq R + \delta \leq 2\delta \]
  and so the larger $w$-disc $|w| \leq R$ is contained in the larger $z$-disc $|z| \leq 2\delta$.

  Then for
  \[ w_1 = z_1 - \delta\frac{z_1}{|z_1|} = z_1\left(1 - \frac{\delta}{|z_1|}\right) \]
  we have
  \[ |w_1| = |z_1|\left| 1 - \frac{\delta}{|z_1|}\right| = -\left(|z_1| - \delta\right) = r. \]
  and so $w_1$ lies on the smaller $w$-disc $|w| = r$. Thus Lemma~\ref{lem2} implies
   \[ |\phi(z_1)| = |\phi_0(w_1)| \leq \frac{2\delta}{|z_1|/2}\frac{2^\kappa c}{|z_1|^\kappa} + \frac{2\delta}{|z_1|/2} \left|\phi\left(\delta\frac{z_1}{|z_1|}\right)\right| \leq \frac{2^{\kappa+2}\delta c}{|z_1|^{\kappa+1}} + \frac{4\delta c_1}{|z_1|} \]
  and the result follows.
\end{proof}
Finally here is the result of Ax that we shall use.

\begin{lemma}\label{Ax-rank-1}
  In characteristic zero let $K$ be a differential field with derivation $D$ and constant field $C$. Let $\xi_1,\ldots,\xi_n$ be in $K$ and let $\hat \xi_1,\ldots,\hat \xi_n$ be in $K^*$ such that:
  \begin{enumerate}[label=(\arabic*)]
    \item\label{diff-eq} $D\xi_i = D\hat \xi_i/\hat \xi_i \ (i=1,\ldots,n)$;
    \item\label{indep-mod-C} $m_1\xi_1 + \cdots + m_n\xi_n$ is not in $C$ for any $m_1,\ldots,m_n$ in $\mathbf{Z}$ not all zero.
  \end{enumerate}
  \noindent
  Then at least $n+1$ among $\xi_1,\ldots,\xi_n,\hat \xi_1,\ldots,\hat \xi_n$ are algebraically independent over $C$.
\end{lemma}

\begin{proof}
  This follows at once from Theorem~3 of \cite[p.\ 253]{Ax1971} (with a single derivation); note that the rank there is 1 because for example $D\xi_n = 0$ would imply $\xi_n$ in $C$ contradicting \ref{indep-mod-C}. Note also that \ref{indep-mod-C} implies that $\xi_1,\ldots,\xi_n$ are all not in $C$, and in particular are all non-zero, which seems to be an additional assumption of this Theorem~3 (and is superfluous anyway).
\end{proof}

In fact we will not need the full force of \cite{Ax1971}, because we apply it with $\xi_1,\ldots,\xi_n$ in the function field $K_0$ of a curve with $D=\de / \de z$ for some (non-constant) $z$ in $K_0$. Then in fact $\hat\xi_1=e^{\xi_1},\ldots,\hat\xi_n=e^{\xi_n}$ are algebraically independent over $K_0$ (so we are rather with what is known in the trade as ``Ax-Lindemann-Weierstrass''). This could be seen directly by considering a relation $\sum_{m=1}^M\beta_me^{\alpha_m}=0$ for $\beta_1,\ldots,\beta_M$ in $K_0$ and $\alpha_1,\ldots,\alpha_M$ in $K_0$ different modulo $\mathbf{C}$, dividing by the last term, differentiating, using induction on $M$ and finally comparing poles of $D\beta_M/\beta_M$ and $D\alpha_M$.

It is convenient to record the following obvious consequence for $n=1$ (``Ax-Hermite-Lindemann'').

\begin{corollary}\label{Ax-HL}
  In characteristic zero let $K$ be a differential field with derivation $D$ and constant field $C$. Let $\xi$ and $\hat\xi \neq 0$ be in $K$ with $D\xi=D\hat\xi/\hat\xi$ and $\xi$ not in $C$. Then $\xi,\hat\xi$ are algebraically independent over $C$.
\end{corollary}

\section{Duality and periods}\label{periods}
Here we record a number of results on functions and differentials, first of all rational and then only meromorphic, on an algebraic curve. Thus let $\mathcal{C}$ be a complete smooth complex algebraic curve of genus $g \geq 0$, and let $S$ be a finite subset of $\mathcal{C}$ with cardinality $s \geq 1$. Denote by $\mathcal{C}_S$ the (affine) curve $\mathcal{C} \setminus S$. By (meromorphic) \emph{differential} on $\mathcal{C}_S$ we mean a differential $1$-form, that is, $\Psi\de \Phi$ for $\Phi,\Psi$ meromorphic on $\mathcal{C}_S$; we call such a differential \emph{regular} on $\mathcal{C}_S$ if it has no poles, and \emph{rational} if we can take $\Psi,\Phi$ rational.

Let $\Delta_S$ denote the space of all rational differentials that are regular on $\mathcal{C}_S$, modulo the ones of the form $\de\phi$ for $\phi$ rational and with no poles on $\mathcal{C}_S$ (we call those \emph{exact}). In other words, $\Delta_S$ is the first cohomology group of the algebraic de Rham complex of $\mathcal{C}_S$.

\begin{lemma}\label{lem:delta-s}
  The linear space $\Delta_S$ has dimension $2g + s - 1$. Furthermore, given some fixed $P_0$ in $S$, each element of $\Delta_S$ can be represented by a differential with at most simple poles in $S \setminus \{P_0\}$.
\end{lemma}
\begin{proof}
  This is fairly well-known, but we show how to recover it easily from the Riemann-Roch theorem. For a given $\mathbf{Z}$-divisor $D$, let $\mathcal{L}(D)$ be the space of all rational functions on $\mathcal{C}$ with divisor at least $-D$, and $\ell(D)$ be its dimension. Let $K_z$ be the divisor of some rational differential form on $\mathcal{C}$, say $\mathrm{d}z$ for some non-constant rational function $z$. Then Riemann-Roch states that
  \begin{equation}\label{eq:RR}
    \ell(D) = \ell(K_z - D) + \deg(D) + 1 - g.
  \end{equation}
  Furthermore, recall that $\ell(K_z-D)$ is also the dimension of the rational differential forms with divisor at least $D$. See for example \cite[p.\ 17]{Ser1988}.

  Fix some $P_0$ in $S$. For positive $m$ sufficiently large, consider the space of differentials with pole at $P_0$ of order at most $m$, and at most simple poles on the rest of $S$ (and no other poles). This has dimension $\ell(K_z-D)$ for
  \[ D=-\sum_{P \in S} P - (m-1)P_0 \]
  and since here $\ell(D)=0$ we get dimension $m + s + g - 2$. And the subspace of the exact ones (which of course have no simple poles) has dimension $\ell((m-1)P_0) - 1 = m - g - 1$. Thus the quotient, which embeds naturally into $\Delta_S$, has dimension $2g + s - 1$.

  We now claim that all other differentials are equivalent to the ones above. Let $\delta$ be a regular differential on $\mathcal{C}_S$. Suppose that $\delta$ has pole of order $k > 1$ at some $P$ in $S \setminus \{P_0\}$. By \eqref{eq:RR}, for every $m$ large enough, there must be a function $h$ in $\mathcal{L}((k - 1)P + mP_0)$ which is not in $\mathcal{L}((k-2)P + mP_0)$, thus with pole of order exactly $k-1$ at $P$ and no other poles except possibly at $P_0$. It follows that $\delta + \alpha\mathrm{d}h$, for some $\alpha$ in $\mathbf{C}$, has pole of order $\leq k-1$ at $P$, and the same poles as $\delta$ on the rest of $S \setminus \{P_0\}$. By an easy induction on the orders of the poles of $\delta$ outside of $P_0$, one finds $f$ such that $\delta + \mathrm{d}f$ has poles of order at most one on $S \setminus \{P_0\}$, as required.
\end{proof}

For $\mathcal{C}$ and $S$ as above, it is well-known that the homology $H_1(\mathcal{C}_S)$ of $\mathcal{C}_S=\mathcal{C}\setminus S$ is free of rank $2g+s-1$ (see for example \cite[p.\ 101]{Ser1988}).

\begin{lemma}\label{lem:non-degen}
  The pairing $H_1(\mathcal{C}_S) \times \Delta_S \to \mathbf{C}$ induced by integration is non-degenerate.
\end{lemma}
\begin{proof}
  Since $H_1(\mathcal{C})$ and $\Delta_S$ have both dimension $2g+s-1$, it suffices to observe the following: if a regular differential $\delta$ on $\mathcal{C}_S$ is such that $\oint_\Gamma\delta = 0$ over every closed path $\Gamma$ on $\mathcal{C}_S$, then $\delta$ is exact, that is, $\delta = \de\phi$ for some rational function $\phi$.

  Let $\delta$ be one such form. Then for instance $\int_{Q_0}^Q\delta$, with $Q_0$ in $\mathcal{C}_S$ fixed, defines a regular function $\phi$ on $\mathcal{C}_S$, and we have $\delta = \de\phi$. Write $\delta=f\mathrm{d}z$, where $f,z$ are rational and $z$ is a local parameter at some $Q$ in $S$. Clearly, if $f$ has a pole of order $k \geq 1$ at $Q$, then for any $\kappa>k-1$ there is $c$ such that $|\phi(P)| \leq  c|z(P)|^{-\kappa}$ for all $P$ in a neighbourhood of $Q$. It follows that $\phi$ extends to a meromorphic function on $\mathcal{C} = \mathcal{C}_S \cup S$, thus $\phi$ is rational by the Riemann Existence Theorem or Chow's Theorem.
\end{proof}

\begin{lemma}\label{lem:rational-diffs}
  Let $\Phi$, $\Psi$ be functions analytic on $\mathcal{C}_S$. Then there is a rational differential $\delta$ regular on $\mathcal{C}_S$ and a function $\phi$ analytic on $\mathcal{C}_S$ such that
  \begin{equation}\label{eq:rational-diffs}
    \Psi\de \Phi\ =\ \delta+\de \phi.
  \end{equation}
\end{lemma}
\begin{proof}
  Consider the periods given by the integrals $\oint_{\Gamma} \Psi\mathrm{d}\Phi$. By Lemma~\ref{lem:non-degen} there is a regular differential $\delta$ on $\mathcal{C}_S$ with exactly the same periods. In particular, $\oint_{\Gamma}\Psi\mathrm{d}\Phi = \oint_{\Gamma}\delta$ for every $\Gamma$ in $H_1(\mathcal{C}_S)$. It follows that
  \[ \phi(Q) = \int_{Q_0}^Q \left(\Psi\mathrm{d}\Phi-\delta\right)\]
  for some fixed $Q_0$ in $\mathcal{C}_S$ defines an analytic function on $\mathcal{C}_S$, hence $\Psi\mathrm{d}\Phi - \delta = \mathrm{d}\phi$, as desired.
\end{proof}

One can also deduce Lemma~\ref{lem:rational-diffs} from Grothendieck's paper \cite{Gro1966} (which is actually an extract from a letter to Atiyah). It relates the complex cohomology (denoted by $H^*(\mathcal{X},\mathbf{C})$ there) to the de Rham cohomology (denoted by $H^*(\mathcal{X},\Omega_\mathcal{X})$ there), even for an algebraic variety $\mathcal{X}$ of any dimension. The proof uses Hironaka's resolution of singularities, which for curves $\mathcal{X}$ is classical.

Further one can describe the pairing in Lemma~\ref{lem:non-degen} more explicitly by looking at period matrices. Enumerate $S = \{P_0, P_1, \ldots, P_{s-1}\}$. We can make a $\mathbf{Z}$-basis
\begin{equation}\label{eq:hom-basis}
  \mathcal{L}_1,\ldots,\mathcal{L}_{2g},\mathcal{M}_1,\ldots,\mathcal{M}_{s-1}
\end{equation}
for $H_1(\mathcal{C}_S)$ out of basis elements $\mathcal{L}_1,\ldots,\mathcal{L}_{2g}$ of $H_1(\mathcal{C})$ and small loops $\mathcal{M}_1,\ldots,\mathcal{M}_{s-1}$ around the points $P_1,\ldots,P_{s-1}$  respectively of $S \setminus\{P_0\}$.

By Lemma~\ref{lem:delta-s} applied to $S_0 = \{P_0\}, S_1 = \{P_0, P_1\}, \dots$ and linear algebra we can make a $\mathbf{C}$-basis
\begin{equation}\label{eq:cohom-basis}
  \rho_1,\ldots,\rho_{2g},\sigma_1,\ldots,\sigma_{s-1}
\end{equation}
for $\Delta_S$ out of differentials of the second kind $\rho_1,\ldots,\rho_{2g}$ (with pole at most in $P_0$ and residue zero) together with $\sigma_1,\ldots,\sigma_{s-1}$ (with simple poles exactly at $P_1,\ldots,P_{s-1}$ respectively and only other pole at $P_0$). Indexing the rows by the differentials and the columns by the loops, we get a period matrix. The block $\Pi$ of size $2g$ in the top left corner corresponds to the period matrix for the case $S = \{P_0\}$, hence it is non-singular, and to its right there is a zero block. The block underneath $\Pi$ we do not know, but to its right we get the diagonal matrix whose diagonal entries are the (non-zero) residues of $\sigma_1,\ldots,\sigma_{s-1}$ at $P_1,\ldots,P_{s-1}$. We discuss in Section~\ref{examples} some examples in which such differentials can be made explicit via well known special functions.

\section{Proof of Theorem~\ref{thm1}}\label{main-proof}
It suffices to prove the result when $\mathcal{V}$ has dimension $n$. In fact the hypotheses imply the dimension is $n$ or $n+1$, and in the latter case we can simply adjoin $\hat X_n=c$.

Let $\mathcal{C}_0$ be the Zariski closure of $\pi(\mathcal{V})$ in $\mathbf{C}^n$, and let $F,G_0$ be as in Lemma~\ref{lem1}. Let $\mathcal{C}$ be a complete smooth model of $\mathcal{C}_0$, so that the coordinate functions $x_1,\ldots,x_n$ on $\mathcal{C}_0$ can be regarded as rational functions $\xi_1,\ldots,\xi_n$ on $\mathcal{C}$. Choose a non-empty finite subset $S$ of $\mathcal{C}$ containing all the poles of $\xi_1,\ldots,\xi_n$ and the zeroes of $G_0(\xi_1,\ldots,\xi_n)$ (which is not identically zero). It will then suffice to find a point of $\mathcal{C}_S = \mathcal{C} \setminus S$ at which the function
\begin{equation}\label{PHI}
  \Phi=F(\xi_1,\ldots,\xi_n,e^{\xi_1},\ldots,e^{\xi_n})
\end{equation}
vanishes.

Note that we can take $S$ arbitrarily large, and this will show that the points $(z_1,\ldots,z_n,e^{z_1},\ldots,e^{z_n})$ of $\mathcal{V}$ project to a set which is Zariski dense in $\mathcal{C}_0$, as mentioned in Section~\ref{intro}.

We assume that there are no such points, and we will reach a contradiction. Thus $\Phi$ does not vanish on $\mathcal{C}_ S$. By Lemma~\ref{lem:rational-diffs} applied to $\Psi=1/\Phi$ we have \eqref{eq:rational-diffs} for some rational differential $\delta$ on $\mathcal{C}$, regular on $\mathcal{C}_ S$, and a function $\phi$, analytic on $\mathcal{C}_ S$.

\begin{claim}\label{claim-phi-rational}
  Under the above assumptions, $\phi$ is rational on $\mathcal{C}$.
\end{claim}
\begin{proof}
  We note that all periods of ${\de \Phi / \Phi}$ are in $2\pi i \bf Z$, because any integral is the variation of $\log \Phi$ continuously along the contour (``Principle of the Argument''). As these are also the periods of $\delta={\de \Phi / \Phi}-\de \phi$, we can define $\exp(\int \delta)$ as a function on $\mathcal{C}_S$, for example as
  \begin{equation}\label{exp}
    \Phi_0(Q) = \exp\left(\int_{Q_0}^Q\delta\right) = \exp\left(\int_{Q_0}^Q\left(\frac{\mathrm{d}\Phi}{\Phi} - \mathrm{d}\phi\right)\right)
  \end{equation}
  for any fixed $Q_0$ in $\mathcal{C}_ S$. Thus ${\de \Phi_0/\Phi_0}=\delta$. Comparing this with \eqref{eq:rational-diffs} we deduce that
  \begin{equation}\label{had}
    \Phi(Q)=c_0\Phi_0(Q)e^{\phi(Q)}, \qquad (c_0=\Phi(Q_0)e^{-\phi(Q_0)} \neq 0)
  \end{equation}
  and from this we will estimate $\phi(Q)$ as $Q$ approaches some point $P$ of $S$.

  Let $z$ be a local parameter at this $P$, so that we can regard everything in \eqref{had} as functions of $z$. Suppose $\xi_1,\ldots,\xi_n$ have poles of orders at most $k \geq 0$ at $P$. Then $|\Phi(Q)| \leq c\exp(c|z(Q)|^{-k})$  by \eqref{PHI} for some $c$ independent of $Q$. Next if we write $\delta=\psi\de z$ then $\psi$ has a pole of order at most $k_0 \geq 0$ at $P$. In \eqref{exp} it is not difficult to see that we can choose the contour to have length bounded independently of $Q$ (near $P$) and with $|\psi| \leq c|z|^{-k_0}$, and it follows that $|\int_{Q_0}^Q\delta| \leq c|z|^{-k_0}$ for some $c$ independent of $Q$ (where for short $z = z(Q)$). Thus $|\Phi_0(Q)|^{-1} = \exp(-\Re\int_{Q_0}^Q\delta) \leq \exp(c|z|^{-k_0})$. So we get similar bounds for $|e^{\phi(Q)}| = |c_0^{-1}\Phi(Q)\Phi_0(Q)^{-1}|$, and it follows that
  \[ \Re \phi(Q) = \log|e^{\phi(Q)}| \leq c|z|^{-\kappa} \]
  for $\kappa=\max\{k,k_0\}$ and some $c$ independent of $Q$. By Lemma~\ref{lem3} we deduce
  \[ |\phi(Q)| \leq c|z|^{-\kappa-1}. \]
  Thus $\phi$ is meromorphic at $P$, and so (as in the proof of Lemma~\ref{lem:non-degen}, say by Riemann Existence) is rational on $\mathcal{C}$ as claimed.
\end{proof}

We can now obtain our contradiction using Ax (Lemma~\ref{Ax-rank-1}) on \eqref{had}. We can identify the function field $K_0=\mathbf{C}(x_1,\ldots,x_n)$ of the affine part of $\mathcal{C}$ with $\mathbf{C}(\xi_1,\ldots,\xi_n)$, and we take ${K}=K_0(\hat\xi_1,\ldots,\hat\xi_n)$ with $\hat\xi_1=e^{\xi_1},\ldots,\hat\xi_n=e^{\xi_n}$.

For a non-constant rational function $z$ on $\mathcal{C}$ we have a derivation $D=\de /\de z$ on $K$ with constant field $C=\bf C$, also acting on $K_0$. Note that \ref{diff-eq} of Lemma~\ref{Ax-rank-1} holds. With $\delta=\psi\de z$ as above we have $D\Phi_0/\Phi_0=\psi$ from \eqref{exp} and so
\begin{equation}\label{EQ}
  D\Phi=\chi\Phi
\end{equation}
for $\chi=\psi+D\phi$ in $K_0$. Writing $F = \sum_{\mathbf{i}}F_{\mathbf{i}}M_{\mathbf{i}}$ for $\mathbf{i}=(i_1,\ldots,i_n)$, where $F_{\mathbf{i}}$ is in $\mathfrak{R}_0=\mathbf{C}[X_1,\ldots,X_n]$ and $M_{\mathbf{i}}=\hat X_1^{i_1}\cdots\hat X_n^{i_n}$, we get $\Phi=\sum_{\mathbf{i}}\gamma_{\mathbf{i}}\Phi_{\mathbf{i}}$ for $\gamma_{\mathbf{i}}$ in $K_0$ and
\[ \Phi_{\mathbf{i}}=\hat \xi_1^{i_1}\cdots\hat \xi_n^{i_n}. \]
Also $D\Phi_\mathbf{i}/\Phi_{\mathbf{i}}=\xi_{\mathbf{i}}$ for
\[ \xi_{\mathbf{i}}=i_1D\xi_1+\cdots+i_nD \xi_n. \]
We find from \eqref{EQ} the equations $\sum_{\mathbf{i}}\beta_{\mathbf{i}}\Phi_{\mathbf{i}}=0$ for
\[ \beta_{\mathbf{i}}=D\gamma_{\mathbf{i}}+\xi_{\mathbf{i}}\gamma_{\mathbf{i}}-\chi\gamma_{\mathbf{i}} \]
also in $K_0$.

If some $\beta_{\mathbf{i}} \neq 0$ this shows that the transcendence degree of $K_0(\hat\xi_1,\ldots,\hat\xi_n)$ over $K_0$ is at most $n-1$, so over $\mathbf{C}$ at most $n$, contradicting Lemma~\ref{Ax-rank-1} (note that \ref{indep-mod-C} there holds because we are assuming that $m_1X_1+\cdots+m_nX_n$ is not constant on $\mathcal V$). Thus we may assume that all $\beta_{\mathbf{i}}=0$.

Next suppose there are two different ${\bf i,i'}$ with
\[ \gamma=\gamma_{\mathbf{i}} \neq 0 \neq \gamma_{\bf i'}=\gamma'. \]
Then
\[ 0=\frac{\beta_{\mathbf{i}}}{\gamma}-\frac{\beta_{\bf i'}}{\gamma'}=\frac{D\gamma}{\gamma}-\frac{D\gamma'}{\gamma'}+(\xi_{\mathbf{i}}-\xi_{\bf i'}), \]
and it follows that $D\xi=D\hat\xi/\hat\xi$ for
\[ \hat\xi=\frac{\gamma'}{\gamma}, \qquad \xi=m_1\xi_1+\cdots+m_n\xi_n \]
with $(m_1,\ldots,m_n)=\mathbf{i}-\mathbf{i}' \neq 0$. Here $\xi,\hat\xi$ are both in $K_0$, so algebraically dependent over $\mathbf{C}$. Thus by Corollary~\ref{Ax-HL} $\xi$ lies in $\mathbf{C}$. However this is also ruled out by our assumption on $m_1X_1+\cdots+m_nX_n$.

Thus there is at most one non-zero $\gamma_{\mathbf{i}}$, and the $F_{\bf i'}~({\bf i'} \neq \mathbf{i})$ are in the prime ideal $\mathfrak P_0$ of $\mathcal{C}$ in $\mathfrak R_0$. But then $F=F_{\mathbf{i}}M_{\mathbf{i}}+\sum_{\mathbf{i}'\neq \mathbf{i}}F_{\bf i'}M_{\bf i'}$ would be in $M_{\mathbf{i}}\mathfrak R_0+\mathfrak P_0\mathfrak R$, excluded in Lemma~\ref{lem1}.

This completes the proof of Theorem~\ref{thm1}.

\section{Proof of Theorem~\ref{thm2}}\label{subcases}
We go through it case-by-case-by-subcase-by-subsubcase.

\subsection*{Case \ref{dim-pi-S-0}} Here $\dim \pi(\mathcal{S}) = 0$, and the conclusion is clear. An example is $X_1=0$, $X_2=0$ with $Z$ as the single point $(0,0,1,1)$.

\subsection*{Case \ref{dim-pi-S-2}} Here $\dim \pi(\mathcal{S}) = 2$, this follows from \cite{BM2017}, even with $Z$ dense in $\mathcal{S}$. An elementary example is $\hat X_1=1,\hat X_2=1$ with $Z=(2\pi i\mathbf{Z})^2 \times \{1\}^2$.

\subsection*{Case \ref{dim-pi-S-1-free}} For $\dim \pi(\mathcal{S}) = 1$ and $\pi(\mathcal{S})$ not contained in a line of rational slope, the conclusion is essentially our Theorem~\ref{thm1} for $n=2$; we just have to recall the remark at the beginning of Section~\ref{main-proof} that the finite set $S$ can be taken arbitrarily large. We have already given some examples but an elementary one is $X_1X_2=1$, $\hat X_1\hat X_2=1$ with $Z$ as the set of all $(z,1/z,e^z,e^{1/z})$ with $z+1/z$ in $2\pi i \mathbf{Z}$.

\medskip

We now consider $\pi(\mathcal{S})$ of dimension $1$ and contained in a line of rational slope $\mathcal{L}$. Say that $\mathcal{L}$ is defined by $m_1X_1 + m_2X_2 = c$ with $m_1,m_2$ integers not both zero and $c$ complex, as in \eqref{LRS}. Then $\mathcal{K}$ is given by
\begin{equation}\label{KRS}
  \hat X_1^{m_1}\hat X_2^{m_2}=\hat c
\end{equation}
with $\hat c=e^c$. Recall that $\mathcal{T}$ from \eqref{GT} is $\mathcal{T} = \mathcal{S} \cap \mathcal{G} = \mathcal{S} \cap (\mathcal{L} \times \mathcal{K})$.

\subsection*{Subcase \ref{dim-T--1} of case \ref{dim-pi-S-1-nonfree}} Here the set $\mathcal{T}$ is empty, and the set $Z$ of $(z_1,z_2,e^{z_1},e^{z_2})$ in $\mathcal S$ lies in $\mathcal{G}$ and so in $\mathcal{T}$, hence $Z$ is empty. We already gave the example \eqref{NZ}.

\subsection*{Subcase \ref{dim-T-2}} Here $\dim\mathcal{T} = 2$, and we can even show that $Z$ is dense in $\mathcal{S}$. Note that in this case $\mathcal{S} = \mathcal{G} = \mathcal{T}$ (because $\mathcal{S}$ and $\mathcal{G}$ are irreducible surfaces), and thus $Z$ coincides with the set of $(z_1,z_2,e^{z_1},e^{z_2})$ with $(z_1,z_2)$ in $\mathcal{L}$. It is now not difficult to see, using the algebraic independence of $z$ and $e^z$, that $Z$ is Zariski dense in $\mathcal{S}$; for example if $m_2 \neq 0$ it contains the set of $(z,z',e^z,e^{z'})$ as $z$ varies in $\bf C$, where $z'=(c-m_1z)/m_2$ (see also the parametrizations \eqref{LP}, \eqref{KP} below). An example is $X_1+X_2=1,\ \hat X_1\hat X_2=e$ with $Z$ as the set of $(z,1-z,e^z,e^{1-z})$ for $z$ in $\mathbf{C}$.

\medskip

Finally, we settle the remaining cases with $\dim\mathcal{T} = 1$.

\subsection*{Subsubcase \ref{dim-pi-T-0} of subcase \ref{dim-T-1}} Here $\dim\pi(\mathcal{T}) = 0$. Since the set $\pi(Z)$ is contained in $\pi(\mathcal{T})$, it must be finite or empty. Thus $Z$ is finite or empty as well.

In fact $Z$ cannot be empty. Namely, for each $Q$ in $\pi(\mathcal{T})$ the fibre $\pi^{-1}(Q)$ in $\mathcal{T}$ has $\dim \pi^{-1}(Q) \leq 1$, and there must be $Q$ with equality. As $\mathcal{T}$ is in $\mathcal{G=L \times K}$, the fibre must be the whole of $Q \times \mathcal{K}$. Writing $Q=(z_1,z_2)$ we see that in particular $(z_1,z_2,e^{z_1},e^{z_2})$ is in the fibre, so in $\mathcal{T}$ and therefore $\mathcal{S}$.

An example is $X_1+X_2=1,~X_1+\hat X_1\hat X_2=e$ with $Z$ as the single point $(0,1,1,e)$. However changing the second equation here to $X_1-X_1^2+\hat X_1\hat X_2=e$ gives an extra point $(1,0,e,1)$ and so $Z$ need not be a single point as in case \ref{dim-pi-S-0}.

\subsection*{Subsubcase \ref{dim-pi-T-1}} This final possibility, $\dim\pi(\mathcal{T}) = 1$, involves the operation of ``back-substitution'' to land in $\mathbf{C} \times \mathbf{C}^*$, so we give the full details. For the moment we assume only the hypotheses in \ref{dim-pi-S-1-nonfree}, that is, $\pi(\mathcal{S})$ has dimension $1$ and is contained in a line $\mathcal{L}$ of rational slope.

We can suppose that $m_1,m_2$ are coprime.  Fix integers $a_1,a_2$ with $a_1m_1+a_2m_2=1$. We can parametrize $\mathcal{L}$ by
\begin{equation}\label{LP}
  X_1=a_1c+m_2Y, \quad X_2=a_2c-m_1Y
\end{equation}
and correspondingly $\mathcal{K}$ by
\begin{equation}\label{KP}
  \hat X_1=\hat c^{a_1}\hat Y^{m_2}, \quad \hat X_2=\hat c^{a_2}\hat Y^{-m_1}.
\end{equation}
By Lemma~\ref{lem1} our $\mathcal{S}$ is defined by $m_1X_1+m_2X_2=c$ and $F=0$. Thus $\mathcal{T}$ is defined by $m_1X_1 + m_2X_2 = c$, \eqref{KRS} and $F=0$. We now check, as mentioned in the introduction, that if $\mathcal{T}$ is non-empty then it is infinite.

For this we define a morphism $f$ from $\mathcal{S}$ to $\mathbf{G}_{\rm m}$ by $f(X_1,X_2,\hat X_1,\hat X_2)=\hat X_1^{m_1}\hat X_2^{m_2}$. If $f$ were not dominant, then $\hat X_1^{m_1}\hat X_2^{m_2}$ would be a constant $\hat c'$ on $\mathcal{S}$. Then $\mathcal{S}$ would be contained in $\mathcal{L} \times \mathcal{K}'$ for some translate $\mathcal{K}'$ of $\mathcal{K}$. As $m_1,m_2$ are coprime, $\mathcal{K}'$ is irreducible; so $\mathcal{S}=\mathcal{L} \times \mathcal{K}'$. Now in fact $\mathcal{K}'=\mathcal{K}$, else $\mathcal{S}$ and $\mathcal{G}=\mathcal{L}\times\mathcal{K}$ would not intersect and $\mathcal{T}$ would be empty. Thus $\mathcal{S}=\mathcal{L} \times \mathcal{K}=\mathcal{G}$ and $\mathcal{T}=\mathcal{S}$ is certainly infinite (and we end up in subcase \ref{dim-T-2}).

Thus we can suppose that $f$ is dominant. Now the Fibre Dimension Theorem \cite[p.\ 228]{DS1998} says that $f^{-1}(\hat c)=\mathcal{T}$ has dimension at least 1; and this finishes the checking.

Define $\tilde{\mathcal{D}}$ in ${\bf C \times C}^*$ as the set of $(Y,\hat Y)$ with $G=0$, where
\[ G=G(Y,\hat Y)=F(a_1c+m_2Y,a_2c-m_1Y,\hat c^{a_1}\hat Y^{m_2},\hat c^{a_2}\hat Y^{-m_1}). \]
Then \eqref{LP}, \eqref{KP} define a map $\varphi$ from $\tilde{\mathcal{D}}$ to $\mathcal{T}$. Its inverse is given by for example
\[ Y=a_2X_1-a_1X_2, \qquad \hat Y=\hat X_1^{a_2}\hat X_2^{-a_1} \]
and so we have isomorphisms.

Therefore returning to our subsubcase \ref{dim-pi-T-1} we must have $\dim \tilde{\mathcal{D}}=1$. In particular $G$ is not identically constant.

However it may not be irreducible; but any irreducible factor gives an irreducible curve. There is at least one of these curves, say $\mathcal{D}$, on which $Y$ is not constant, else $Y$ would take at most finitely many values on $\tilde{\mathcal{D}}$ and then $(X_1,X_2)$ would take at most finitely many values in $\phi(\tilde{\mathcal{D}})=\mathcal{T}$ by \eqref{KRS}, contrary to our assumption $\dim \pi(\mathcal{T})=1$. Therefore by Theorem~\ref{thm1} the set $W$ of points $(w,e^w)$ on $\mathcal{D}$ project to a Zariski dense subset of $\mathbf{C}$. It follows that $\pi(\varphi(W))$ is Zariski dense in $\pi(\mathcal{S})$. Finally $\varphi(W)$ is contained in $Z$ by \eqref{LP} and \eqref{KP}.

An example is \eqref{EX}; the choice $a_1=0$, $a_2=1$ leads to $Y=X_1$, $\hat Y=\hat X_1$ and so \eqref{EXX}.

\section{Examples}\label{examples}
We can actually go further with \eqref{had} in the style $e^z+e^{1/z}-1=z^me^{\phi(z)}$ of \eqref{HFT}. Namely the period matrix just after \eqref{eq:cohom-basis} is some invertible $M$. Thus we can act on \eqref{eq:cohom-basis} by $2\pi iM^{-1}$ to get $2\pi i I_h$ for the identity matrix of order $h=2g+s-1$, and then integrating these and exponentiating as in \eqref{exp} gives $\Phi_1,\ldots,\Phi_h$ analytic on $\mathcal{C}_S$ and never vanishing there. Now the period (row) vector of $\de \Phi/\Phi$ is $2\pi i\mathbf{m}$ for some $\mathbf{m}=(m_1,\ldots,m_h)$ in $\mathbf{Z}^h$, and so we find
\begin{equation}\label{ghad}\Phi=\Phi_1^{m_1}\cdots\Phi_h^{m_h}e^\phi
\end{equation}
in \eqref{had}.

Originally we proved \eqref{ghad} for small genus actually by constructing $\Phi_1,\ldots,\Phi_h$ directly in an \emph{ad hoc} fashion. As some amusing formulae turned up we feel it may be of some interest to present our constructions here.

\subsection*{Case $g = 0$} Now $\mathcal{C}$ may be taken as $\mathbf{P}_1$, which we identify with $\mathbf{C} \cup \{\infty\}$.

If $s=1$ then $h=0$ and there is nothing to do. So we assume $s \geq 2$.

If $S$ contains $\infty$ then $S=\{\infty,p_1,\ldots,p_{s-1}\}$, and clearly $z-p_1,\ldots,z-p_{s-1}$ are the desired $\Phi_i$'s. And indeed we find $\de (z-p_i)/(z-p_i)=\de z/(z-p_i)=\sigma_i$ satisfying exactly the same conditions as in the original \eqref{eq:cohom-basis}. The period matrix in Lemma~\ref{lem:non-degen}, with the homology basis \eqref{eq:hom-basis}, is $2\pi iI_{s-1}$  (provided we choose the appropriate orientations). Thus $\delta$ in Lemma~\ref{lem:rational-diffs} for general $\Phi$ must have a decomposition
\[ \delta=m_1\sigma_1+\cdots+m_{s-1}\sigma_{s-1} \]
now with integer coefficients. And so $\Phi_0(z)=(z-p_1)^{m_1}\cdots(z-p_{s-1})^{m_{s-1}}$ in the proof of Claim~\ref{claim-phi-rational}. Thus if $\Phi$ is any function analytic on $\mathbf{C}_S$ and not vanishing there, it has the form
\[ \Phi(z)=(z-p_1)^{m_1}\cdots(z-p_{s-1})^{m_{s-1}}e^{\phi(z)} \]
for some $\phi$ also analytic on $\mathbf{C}_S$, exactly as in Example~\ref{ex1}. The special case $s=2$ and $p_1=0$ is \eqref{HFT}.

If $S=\{p_0,p_1,\ldots,p_{s-1}\}$ does not contain $\infty$, then we can use in a similar way for example
\[ {z-p_1 \over z-p_0},\ldots,{z-p_{s-1} \over z-p_0}. \]

\subsection*{Case $g = 1$} Now $\mathcal{C}$ can be taken as an elliptic curve $\mathcal{E}$, with origin $O$, whose affine part is
\[ y^2=4x^3-g_2x-g_3. \]
It is parametrized by the Weierstrass functions $x=\wp(z),y=\wp'(z)$ with corresponding period lattice $\Omega$.

Now examples of $\Phi$ are not so easy to write analytically as $z-p$; but we found the following, at first for $S=\{O\}$. Take any period $\omega$ in $\Omega$. It has a corresponding quasi-period $\eta$ defined by $\zeta(z+\omega)=\zeta(z)+\eta$ for the associated Weierstrass zeta function. Then
\begin{equation}\label{zba}
  \Phi^{(\omega)}(z)=e^{\omega\zeta(z)-\eta z}
\end{equation}
is doubly periodic. This is because
\[ \omega\zeta(z+\tilde\omega)-\eta(z+\tilde \omega)=\omega\zeta(z)-\eta z+\omega\tilde\eta-\eta\tilde\omega \]
for any other period $\tilde\omega$ with quasi-period $\tilde\eta$; and the Legendre relations show that $\omega\tilde\eta-\eta\tilde\omega$ is in $2\pi i\bf Z$. In fact we have here the very simplest form of a Baker-Akhiezer function, with an essential singularity at $z=0$; see \cite{Bak1995} Chapter XIV and in particular page xxviii of Krichever's foreword for more general versions, although this particular example does occur, even for arbitrary genus, in Weierstrass \cite[p.\ 312]{Wei1967}. It clearly does not vanish on $\mathcal{E}_S$ for $S=\{O\}$.

See also Pellarin \cite{Pel2013} for analogues of Baker-Akhiezer in positive characteristic.

In fact we get two for the price of one by taking basis elements of $\Omega=\mathbf{Z}\omega_1+\mathbf{Z}\omega_2$ and corresponding $\eta_1$, $\eta_2$ so $\Phi_1$, $\Phi_2$ (note that these depend on the choice of basis, but the multiplicative group they generate does not). As $\zeta'=-\wp$ we find
\[ {\de \Phi_i \over \Phi_i}=-(\omega_i\wp(z)+\eta_i)\de z=-(\omega_ix+\eta_i){\de x \over y} \qquad (i=1,2). \]
There are loops $\mathcal{L}_1,\mathcal{L}_2$ with
\[ \int_{\mathcal{L}_i}{\de x \over y}=\omega_i, \quad \int_{\mathcal{L}_i}{x\de x \over y}=-\eta_i \qquad (i=1,2) \]
(of course the differentials here are possible $\rho_1,\rho_2$ in the modified \eqref{eq:cohom-basis} above) and so the periods we want are
\[ \int_{\mathcal{L}_1}{\de \Phi_1 \over \Phi_1}=0, \quad \int_{\mathcal{L}_2}{\de \Phi_1 \over \Phi_1}=\pm 2\pi i \]
\[ \int_{\mathcal{L}_1}{\de \Phi_2 \over \Phi_2}=\pm 2\pi i, \quad \int_{\mathcal{L}_2}{\de \Phi_2 \over \Phi_2}=0 \]
by the more precise form of Legendre (depending on orientation). This does not quite give $2\pi iI_h=2\pi iI_2$ for the case $S=\{O\}$; but at least we get $2\pi iU$ for unimodular $U$ with $\det U=\pm 1$. This is already enough to imply that any function $\Phi$, analytic on $\mathcal{E}_S$ and never vanishing there, has the form $\Phi=\Phi_1^{m_1}\Phi_2^{m_2}e^\phi$ for $\phi$ also analytic on $\mathcal{E}_S$; which is equivalent to Example~\ref{ex2}.

For more general $S=\{O,P_1,\ldots,P_{s-1}\}$ containing $O$ we found other examples as follows. Suppose $P \neq O$; then we can write $P=(\wp(u),\wp'(u))$ and define
\begin{equation}\label{sig}
  \Psi^{(u)}(z)={\sigma(z-u) \over \sigma(z)}e^{u\zeta(z)}
\end{equation}
for the Weierstrass sigma function (note that this depends on the choice of $u$ but if we change $u$ by a period $\omega$ then $\Psi^{(u)}$ changes by $\Phi^{(\omega)}$ up to constants). It is analytic on $\mathcal{E}$ with $O,P$ removed and never vanishes there. It too is a Baker-Akhiezer function (and almost certainly known to Weierstrass) with its essential singularity at $z=0$. That distinguishes it from a similar expression occurring in the exponential map for a multiplicative extension of $\mathcal{E}$, which has $e^{\zeta(u)z}$ in place of $e^{u\zeta(z)}$.

As $\sigma'/\sigma=\zeta$ we obtain
\begin{equation}\label{wp2}
  {\de \Psi^{(u)} \over \Psi^{(u)}}=(\zeta(z-u)-\zeta(z)-u\wp(z))\de z,
\end{equation}
which by the addition theorem for $\zeta$ is
\[ \left(-\zeta(u)+{1 \over 2}{\wp'(z)+\wp'(u) \over \wp(z)-\wp(u)}-u\wp(z)\right)\de z=-\zeta(u){\de x \over y}-u{x\de x \over y}+\theta_P \]
for the perhaps more classically familiar
\[ \theta_P={1 \over 2}{y+\wp'(u) \over x-\wp(u)}{\de x \over y}. \]
Both have residue divisor $P-O$.

Thus with $P=P_1,\ldots,P_{s-1}$ (of course giving rise to $\sigma_1,\ldots,\sigma_{s-1}$ in \eqref{eq:cohom-basis} above) and $\mathcal{M}_1,\ldots,\mathcal{M}_{s-1}$ as in the proof of Lemma~\ref{lem:non-degen} we obtain in the bottom right block of the period matrix a diagonal matrix with entries $\pm 2\pi i$. So again something unimodular and corresponding $\Psi_1,\ldots,\Psi_{s-1}$.

Therefore any function $\Phi$, analytic on $\mathcal{E}_S$ and never vanishing there, has the form
\begin{equation}\label{ex3}
  \Phi=\Phi_1^{m_1}\Phi_2^{m_2}\Psi_1^{n_1}\cdots\Psi_{s-1}^{n_{s-1}}e^\phi
\end{equation}
for $\phi$ also analytic on $\mathcal{E}_S$.

Occasionally we can find simpler $\Psi$. For example if $P=(e,0)$ is a point of order 2 then $\Psi=x-e=\wp(z)-e$ is analytic on $\mathcal{E}_S$ for $S=\{O,P\}$ and never vanishes there (even without essential singularity). And if say $e=\wp(\omega/2)$ then indeed
\[ \wp(z)-e={(\Psi^{(\omega/2)}(z))^2 \over (\sigma(\omega/2))^2\Phi^{(\omega)}(z)}. \]

And if our $S=\{P_0,P_1,\ldots,P_{s-1}\}$ does not contain $O$ then we can simply use the group law to reduce to $\{O,P_1-P_0,\ldots,P_{s-1}-P_0\}$ thus obtaining for example $\Phi^{(\omega)}(z-u_0)$ in place of \eqref{zba}.

\subsection*{Case $g = 2$} Here it will suffice to deal with a complex hyperelliptic curve $\mathcal{H}$ whose affine part is defined by
\begin{equation}\label{e2}y^2=x^5+b_1x^4+b_2x^3+b_3x^2+b_4x+b_5
\end{equation}
with the discriminant of the right-hand side non-zero. We are therefore using the notation of Grant \cite{Gra1990}. Of course there is no longer a parametrization by $\mathbf{C}$. To obtain the analogue of $\wp$ and so on we must embed $\mathcal{H}$ into its Jacobian, which is parametrized by $\mathbf{C}^2$.

We originally constructed examples of $\Phi$ using theta functions. We fix a matrix $T=\begin{pmatrix} \tau_{1}&\tau \\ \tau&\tau_{2}\end{pmatrix}$ in the Siegel upper half space. We have a standard theta function $\theta(\mathbf{z})$ defined for $\mathbf{z}=\begin{pmatrix} z_1\\ z_2 \end{pmatrix}$ by
\[ \theta(\mathbf{z})=\sum_{\mathbf{p}\in \mathbf{Z}^2}\exp(\pi i (\mathbf{p}^tT\mathbf{p}+2\mathbf{p}^t\mathbf{z})) \]
with column vectors $\mathbf{p}$. It satisfies
\[ \theta({\bf z+e}_1)=\theta(\mathbf{z}), \quad \theta({\bf z+e}_2)=\theta(\mathbf{z}) \]
for $\mathbf{e}_1=\begin{pmatrix}1\\ 0\end{pmatrix},~\mathbf{e}_2=\begin{pmatrix}0\\ 1\end{pmatrix}$, as well as
\[ \theta({\bf z+t}_1)=c_1\exp(-2\pi i z_1)\theta(\mathbf{z}), \quad \theta({\bf z+t}_2)=c_2\exp(-2\pi i z_2)\theta(\mathbf{z}) \]
for $\mathbf{t}_1=\begin{pmatrix}\tau_1\\ \tau\end{pmatrix},~\mathbf{t}_2=\begin{pmatrix}\tau\\ \tau_2\end{pmatrix}$ and constants $c_1$, $c_2$ (see for example \cite[pp.\ 118--120]{Mum1983}).

So the ``Baker zeta functions''
\[ \zeta_1={1 \over \theta}{\partial \theta \over \partial z_1}, \quad \zeta_2={1 \over \theta}{\partial \theta \over \partial z_2} \]
satisfy
\begin{equation}\label{qper1}
  \zeta_1({\bf z+e}_1)=\zeta_1(\mathbf{z}),~\zeta_1({\bf z+e}_2)=\zeta_1(\mathbf{z}),~\zeta_1({\bf z+t}_1)=-2\pi i+\zeta_1(\mathbf{z}),~\zeta_1({\bf z+t}_2)=\zeta_1(\mathbf{z}),
\end{equation}
\begin{equation}\label{qper2}
  \zeta_2({\bf z+e}_1)=\zeta_2(\mathbf{z}),~\zeta_2({\bf z+e}_2)=\zeta_2(\mathbf{z}),~\zeta_2({\bf z+t}_1)=\zeta_2(\mathbf{z}),~\zeta_2({\bf z+t}_2)=-2\pi i+\zeta_2(\mathbf{z}).
\end{equation}
So miraculously
\begin{equation}\label{ft}\Phi_1=e^{\zeta_1}, \quad \Phi_2=e^{\zeta_2}
\end{equation}
are ``quadruply periodic'' (also known to Weierstrass). Suitably translated, these generalize one of the $\Phi^{(\omega)}=e^{\omega\zeta(z)-\eta z}$; in fact here $\omega=1$ and $\eta=0$ in the new normalization.

We need another pair corresponding to $\omega=\tau$. A short calculation shows that
\[ \zeta_3=\tau_1\zeta_1+\tau\zeta_2+2\pi iz_1, \quad \zeta_4=\tau\zeta_1+\tau_2\zeta_2+2\pi iz_2 \]
satisfy
\begin{equation}\label{qper3}
  \zeta_3({\bf z+e}_1)=\zeta_3(\mathbf{z})+2\pi i,~\zeta_3({\bf z+e}_2)=\zeta_3(\mathbf{z}),~\zeta_3({\bf z+t}_1)=\zeta_3(\mathbf{z}),~\zeta_3({\bf z+t}_2)=\zeta_3(\mathbf{z}),
\end{equation}
\begin{equation}\label{qper4}
  \zeta_4({\bf z+e}_1)=\zeta_4(\mathbf{z}),~\zeta_4({\bf z+e}_2)=\zeta_4(\mathbf{z})+2\pi i,~\zeta_4({\bf z+t}_1)=\zeta_4(\mathbf{z}),~\zeta_4({\bf z+t}_2)=\zeta_4(\mathbf{z}).
\end{equation}
Thus
\begin{equation}\label{lt}
  \Phi_3=e^{\tau_1\zeta_1+\tau\zeta_2+2\pi iz_1}, \quad \Phi_4=e^{\tau\zeta_1+\tau_2\zeta_2+2\pi iz_2}
\end{equation}
will do (known of course to Weierstrass).

These are functions on open subsets of $\mathbf{C}^2$, and we get functions on $\mathcal{H}$ by taking restrictions to a one-dimensional analytic set. So we have to understand their poles, which just come from the zeroes of $\theta$. In the usual notation we choose basis elements $\mathcal{A}_1,\mathcal{A}_2,\mathcal{B}_1,\mathcal{B}_2$ (aka $\mathcal{L}_1,\mathcal{L}_2,\mathcal{L}_3,\mathcal{L}_4$ in \eqref{eq:hom-basis} above) for the homology of $\mathcal{H}$ in the standard way and then differentials $\rho_1,\rho_2$ of the first kind on $\mathcal{H}$ normalized such that the respective integrals of $\begin{pmatrix}\rho_1\\ \rho_2\end{pmatrix}$ are the columns $\mathbf{e}_1,\mathbf{e}_2,\mathbf{t}_1,\mathbf{t}_2$. Then
\[ \varepsilon(Q)=\begin{pmatrix}\int_\infty^Q\rho_1\\ \int_\infty^Q\rho_2\end{pmatrix} \]
embeds $\mathcal{H}$ into $\mathbf{C}^2/\Lambda$ for the lattice generated by these columns.

The Riemann Vanishing Theorem implies that there is some $\mathbf{u}_0$ such that $\theta(\mathbf{z})=0$ if and only if there is $Q$ in $\mathcal{H}$ with $\mathbf{z}=\mathbf{u}_0-\varepsilon(Q)$ modulo $\Lambda$ (see for example Corollary~3.6 of \cite[p.\ 160]{Mum1983}). In particular $\theta(\mathbf{u}_0)=0$ (in fact because our $\mathcal{H}$ is hyperelliptic we have $\mathbf{u}_0=\mathbf{e}_1+{1 \over 2}\mathbf{e}_2+{1 \over 2}\mathbf{t}_1+{1 \over 2}\mathbf{t}_2$, in ${1 \over 2}\Lambda$  but not $\Lambda$ -- see for example \cite[3.80, 3.82]{Mum1984} -- however if we wanted to progress to curves of genus $g \geq 3$ which are not hyperelliptic then we should forget this explicit value).

Now the trouble with \eqref{ft} and \eqref{lt} is that theta functions tend to have two zeroes when restricted to $\varepsilon(\mathcal{H})$ (or even infinitely many, such as $\theta(\mathbf{u}_0-\varepsilon(Q))$ for example). We can overcome this problem for $S=\{P_0\}$ provided $P_0$ is not one of the six Weierstrass points, which are $\infty$ and the five points with $y=0$ on \eqref{e2}. Namely consider
\[ \lambda^{(P_0)}(Q)=\theta(\mathbf{u}_0-2\varepsilon(P_0)+\varepsilon(Q)). \]
This vanishes if and only if there is $Q'$ in $\mathcal{H}$ with $\mathbf{u}_0-2\varepsilon(P_0)+\varepsilon(Q)=\mathbf{u}_0-\varepsilon(Q')$ modulo $\Lambda$; that is
\[ \varepsilon(Q)+\varepsilon(Q')=\varepsilon(P_0)+\varepsilon(P_0) \mod \Lambda. \]
If at least one of $Q,Q'$ is not $P_0$, then by Abel-Jacobi there is a rational function on $\mathcal{H}$ with a double or single pole at $P_0$ and no other poles. However by the definition of Weierstrass point that is impossible. Thus $Q=P_0\ (=Q')$ and in particular $\lambda^{(P_0)}$ is not identically zero on $\mathcal{H}$.

Thus \eqref{ft} and \eqref{lt} restricted to $\mathbf{z}=\mathbf{u}_0-2\varepsilon(P_0)+\varepsilon(Q)$ provide functions
\[ \Phi_1^{(P_0)},\ \Phi_2^{(P_0)},\ \Phi_3^{(P_0)},\ \Phi_4^{(P_0)} \]
analytic on $\mathcal{H}_S$ (for this singleton $S$) never vanishing there (also Weierstrass-Baker-Akhiezer).

As for the periods, we have for example $\de \Phi_1/\Phi_1=\de \zeta_1$ and so the integral of $\Phi_1^{(P_0)}$ around say $\mathcal{A}_1$ corresponds to the change in $\zeta_1(\mathbf{u}_0-2\varepsilon(P_0)+\varepsilon(Q))$, which is zero by the first of \eqref{qper1}. So we see that the periods of $\de \Phi_1^{(P_0)}/\Phi_1^{(P_0)}$ are
\[ 0,\ 0,\ \pm2\pi i,\ 0; \]
(depending on orientation) and likewise from \eqref{qper2} the periods of $\de \Phi_2^{(P_0)}/\Phi_2^{(P_0)}$ are
\[ 0,\ 0,\ 0,\ \pm2\pi i. \]
And from \eqref{qper3}, \eqref{qper4} we find that the periods of $\de \Phi_3^{(P_0)}/\Phi_3^{(P_0)},\de \Phi_4^{(P_0)}/\Phi_4^{(P_0)}$ are
\[ \pm2\pi i,\ 0,\ 0,\ 0 \]
\[ 0,\ \pm2\pi i,\ 0,\ 0 \]
respectively. So once again we get a matrix $M_0=2\pi iU$ for unimodular $U$. Thus any function $\Phi$, analytic on $\mathcal{H}_S$ for this $S=\{P_0\}$ and never vanishing there, has the form
\begin{equation}\label{ex4}
  \Phi=(\Phi_1^{(P_0)})^{m_1}(\Phi_2^{(P_0)})^{m_2}(\Phi_3^{(P_0)})^{m_3}(\Phi_4^{(P_0)})^{m_4}e^\phi
\end{equation}
for $\phi$ also analytic on $\mathcal{H}_S$.

But what about $S=\{P_0\}$ for a Weierstrass point $P_0$? The above fails because $\lambda^{(P_0)}$ is then identically zero on $\mathcal{H}$ thanks to the function $x$ or $1/(x-e)$ for the zeroes $e$ of the right-hand side of \eqref{e2}.

If $P_0=\infty$ for example, then we could try to flip back into the construction in Section~\ref{periods}, because $\rho_1,\rho_2,\rho_3,\rho_4$ can be taken as the well-known
\begin{equation}\label{dsk}
  {\de x \over y},~{x\de x \over y},~{x^2\de x \over y},~{x^3\de x \over y}.
\end{equation}
The right linear combinations appear to be connected to the ``Legendre relations'' of \cite[p.\ 14]{Bak1907}, and then one would have to integrate and exponentiate. However the analogue of \eqref{dsk} for a Weierstrass point $P_0 \neq \infty$ seems messy.

Alternatively here is a dirty trick that works for any Weierstrass point $P_0$. We choose any $Q_0$ not a Weierstrass point and as above
\[ \lambda^{(Q_0)}(Q)=\theta(\mathbf{u}_0-2\varepsilon(Q_0)+\varepsilon(Q)) \]
has a double zero at $Q=Q_0$ and no other zero. Similar arguments show that
\[ \mu^{(Q_0)}(Q)=\theta(\mathbf{u}_0-\varepsilon(Q_0)-\varepsilon(P_0)+\varepsilon(Q)) \]
has simple zeroes at $Q=Q_0,P_0$ and no other zero. Thus $(\mu^{(Q_0)})^2/\lambda^{(Q_0)}$ has a double zero at $Q=P_0$ and no other zeroes or poles. And so
\[ \exp\bigl(2\zeta_i(\mathbf{u}_0-\varepsilon(Q_0)-\varepsilon(P_0)+\mathbf{z})-\zeta_i(\mathbf{u}_0-2\varepsilon(Q_0)+\mathbf{z})\bigr) \qquad (i=1,2) \]
are the analogues of \eqref{ft} for example.

And the new period matrix is just $2M_0-M_0=M_0$ the old period matrix.

This settles \eqref{ex4} for singletons $S=\{P_0\}$. For general $S=\{P_0,P_1,\ldots,P_{s-1}\}~(s \geq 2)$ we write down the analogue of \eqref{sig} as
\begin{equation}\label{sig2}\Psi^{(\mathbf{u})}(\mathbf{z})={\theta({\bf z-u}) \over \theta(\mathbf{z})}e^{u_1\zeta_1(\mathbf{z})+u_2\zeta_2(\mathbf{z})}, \qquad \mathbf{u}=\begin{pmatrix}u_1\\ u_2\end{pmatrix},
\end{equation}
also quadruply periodic. If $P \neq \infty$ then
\[ \theta(\mathbf{u}_0-\varepsilon(P)+\varepsilon(Q)) \]
vanishes at $Q=P,\infty$ and nowhere else. So if also $P_0 \neq \infty$ then taking $\mathbf{z}=\mathbf{u}_0-\varepsilon(P_0)+\varepsilon(Q)$ and $\mathbf{u}=\varepsilon(P)-\varepsilon(P_0)$ in \eqref{sig2} we get a function $\Psi^{(P_0,P)}$ analytic on $\mathcal{H}$ with $P_0$, $P$ removed and never vanishing there (also Baker-Akhiezer).

Then $\de \Psi^{(P_0,P)}/\Psi^{(P_0,P)}$ is a rational differential on $\mathcal{H}$ whose residue divisor is the divisor $P-P_0$ of $\Psi^{(P_0,P)}$ itself, and so integrating over a small loop around $P$ gives $\pm 2\pi i$. In fact it is only $\theta({\bf z-u})$ in \eqref{sig2} that causes the pole of the differential at $P$.

One could go further as in \eqref{wp2} by using the ``Baker $\wp$ functions''
\[ \wp_{ij}=-{\partial \zeta_i \over \partial z_j} \qquad (i,j=1,2) \]
and even $\wp_{ijk}$ (not quite as in \cite[p.\ 38]{Bak1907} or \cite[p.\ 99]{Gra1990}), and this would lead to the corresponding differentials \eqref{eq:cohom-basis} on $\mathcal{H}$.

Anyway, doing the above for $P=P_1,\ldots,P_{s-1}$ leads to the usual unimodular matrix, so we conclude that if $S=\{P_0,P_1,\ldots,P_{s-1}\}$ does not contain $\infty$ then any function $\Phi$, analytic on $\mathcal{H}_S$ and never vanishing there, has the form
\begin{equation}\label{ex5}
  \Phi=\left(\Phi_1^{\left(P_0\right)}\right)^{m_1}\left(\Phi_2^{\left(P_0\right)}\right)^{m_2}\left(\Phi_3^{\left(P_0\right)}\right)^{m_3}\left(\Phi_4^{\left(P_0\right)}\right)^{m_4}\left(\Psi^{\left(P_0,P_1\right)}\right)^{n_1}\cdots\left(\Psi^{\left(P_0,P_{s-1}\right)}\right)^{n_{s-1}}e^\phi
\end{equation}
for $\phi$ also analytic on $\mathcal{H}_S$.

Presumably $\infty$ can be handled with more dirty tricks.

Probably these constructions extend to any genus $g \geq 3$, at least if $P_0,P_1,\ldots,P_{s-1}$ are in ``general position''.

\section{Effectivity}\label{effectivity}
Here we sketch some possibilities for effective versons of Theorem~\ref{thm1}. For simplicity we restrict ourselves to genus $g=0$.

The basic idea can be illustrated with $n=1$ and $\mathcal{V}$ in $\mathbf{C} \times \mathbf{C}^*$ defined by $X_1=\hat X_1$; so we want a zero of $\Phi(z)=e^z-z$. We will localize in the sense of finding some explicit $R$ such that there is a zero $z_0$ with $|z_0| \leq R$. Of course there are many direct ways of doing this, but they do not extend to general $n$.

So let us suppose to the contrary that for some $R>0$ there is no $z_0$ with $\Phi(z_0)=0$ and $|z_0| \leq R$. Now the (classical) argument shows that $\Phi=e^\phi$ for some $\phi$ analytic on the disc $|z| \leq R$. So
\[ \Re \phi(z) = \log|e^z-z| \leq \log(e^R+R) \]
on this disc.

As $e^{\phi(0)}=\Phi(0)=1$ we can assume $\phi(0)=0$. Then Lemma~\ref{lem2} with $r=R/2$ gives
\[ \sup_{|z| \leq R/2}|\phi(z)| \leq 2\log(e^R+R). \]

This says that $|\phi|=O(R)$ on ``large'' discs; thus $\phi$ ought to be ``almost'' a linear polynomial $az+b$. More precisely if $\phi(z)=\sum_{k=0}^\infty a_kz^k$ then we could show that the $a_k~(k=2,3,\ldots)$ are ``small'', so that $\Phi(z)$ is ``near'' $e^{az+b}$. As one might guess from \eqref{ab} in the discussion of Section~\ref{intro}, this could be disproved with an appropriate ``effective'' extension of Ax's Theorem. In fact such an extension can be supplied; however here we can take a short cut as follows.

We look at just
\begin{equation}\label{aest}
  |a_2|=\left|{1 \over 2\pi i}\int_{|z|=R/2}{\phi(z) \over z^3}\de z\right| \leq 8{\log(e^R+R) \over R^2}.
\end{equation}
On the other hand (recall $a_0=0$)
\begin{equation}\label{tay}
  1+{1 \over 2}z^2+\cdots~=~e^z-z~=~e^{\phi(z)}~=~1+a_1z+\left(a_2+{1 \over 2}a_1^2\right)z^2+\cdots
\end{equation}
and we deduce $a_1=0$ and $a_2=1/2$. This contradicts \eqref{aest} for $R=17$.

For $n=2$ and say $X_1X_2=1,~\hat X_1+\hat X_2=1$ we propose to find $R$ such that there is a zero $z_0$ of $\Phi(z)=e^z+e^{1/z}-1$ with $1/R \leq |z_0| \leq R$. Now we have to use Laurent series, and we can show that $\Phi(z)=z^me^{\phi(z)}$ holds as in \eqref{HFT} with $\phi(z)=\sum_{k=-\infty}^\infty a_kz^k$, and that $a_k$ for $k \geq 3$ and $k \leq -3$ are ``small''. Thus $\phi$ is near $az^2+bz+c+d/z+e/z^2$ as in \eqref{lau}. This time a short cut by equating coefficients does not seem so easy as in \eqref{tay}, but ``effective Ax'' can be used, or also repeated differentiation (here five times suffice) to deduce a contradiction for large $R$; provided we can also estimate the exponent $m$.

This latter problem seems not entirely trivial. If we take $R>1$, then the method shows that
\begin{equation}\label{wind}
  m=\pm{1 \over 2\pi i}\int_{|z|=1}{\de \Phi \over \Phi}.
\end{equation}
We politely asked Maple to compute this but it refused to answer. Finally we realized that this was due to zeroes $z_0$ with $|z_0|=1$. Writing $z=e^{i t}$, drawing a graph to guess a rough solution and then refining with Newton gives indeed the pair
\[ -0.08285557733006468223\ldots\pm.9965615652358371338\ldots i. \]
So we did actually stumble on zeroes!

Incidentally, this leads to a one-sentence proof that there is a zero, because $\Phi(e^{it})$ is continuous from $\bf R$ to $\bf R$ with value $2e-1>0$ at $t=0$ and value $2e^{-1}-1<0$ at $t=\pi$.

If $R>2$ and for some reason we had taken say $|z|=2$ in \eqref{wind} then Maple would have obliged with something looking suspiciously like $m=1$ (and other more theoretical considerations lead to a rigorous bound -- we found $|m| \leq 62$ for example). Such shady calculations suggest that there is a pair of zeroes with $7<|z_0|<8$ (our own value for $R$ was about $10^8$).

\bibliographystyle{amsplain}
\bibliography{References}

\providecommand{\bysame}{\leavevmode\hbox to3em{\hrulefill}\thinspace}
\providecommand{\MR}{\relax\ifhmode\unskip\space\fi MR }
\providecommand{\MRhref}[2]{%
  \href{http://www.ams.org/mathscinet-getitem?mr=#1}{#2}
}
\providecommand{\href}[2]{#2}
\begin{thebibliography}{10}

\bibitem{Ax1971}
J.~Ax, \emph{On {{Schanuel}}'s conjectures}, Annals of Mathematics. Second
  Series \textbf{93} (1971), 252--268.

\bibitem{Bak1907}
H.~F. Baker, \emph{An introduction to the theory of multiply periodic
  functions}, {Cambridge University Press}, {Cambridge}, 1907.

\bibitem{Bak1995}
\bysame, \emph{Abelian functions: {{Abel}}'s theorem and the allied theory of
  theta functions}, Cambridge {{Mathematical Library}}, {Cambridge University
  Press}, {Cambridge}, 1995.

\bibitem{BK2018}
M.~Bays and J.~Kirby, \emph{Pseudo-exponential maps, variants, and
  quasiminimality}, Algebra {$\&$} Number Theory \textbf{12} (2018), no.~3,
  493--549.

\bibitem{BM2017}
W.~D. Brownawell and D.~W. Masser, \emph{Zero estimates with moving targets},
  Journal of the London Mathematical Society \textbf{95} (2017), no.~2,
  441--454.

\bibitem{DS1998}
V.~I. Danilov and V.~V. Shokurov, \emph{Algebraic curves, algebraic manifolds
  and schemes}, Encyclopedia of {{Mathematical Sciences}}, vol.~23,
  {Springer-Verlag}, {Berlin}, 1998.

\bibitem{DFT2018}
P.~D'Aquino, A.~Fornasiero, and G.~Terzo, \emph{Generic solutions of equations
  with iterated exponentials}, Transactions of the American Mathematical
  Society \textbf{370} (2018), no.~2, 1393--1407.

\bibitem{DFT2021}
\bysame, \emph{A weak version of the strong exponential closure}, Israel
  Journal of Mathematics \textbf{242} (2021), no.~2, 697--705.

\bibitem{Dav1973}
M.~Davis, \emph{Hilbert's tenth problem is unsolvable}, Amer. Math. Monthly
  \textbf{80} (1973), 233--269. \MR{317916}

\bibitem{GMS1963}
I.~J. Good, A.~J. Mayne, and J.~M. Smith (eds.), \emph{The scientist
  speculates: An anthology of partly-baked ideas}, {Basic Books Inc.}, {New
  York}, 1963.

\bibitem{Gra1990}
D.~Grant, \emph{Formal groups in genus two}, Journal f\"ur die Reine und
  Angewandte Mathematik. [Crelle's Journal] \textbf{411} (1990), 96--121.

\bibitem{Gro1966}
A.~Grothendieck, \emph{On the de {{Rham}} cohomology of algebraic varieties},
  Institut des Hautes \'Etudes Scientifiques. Publications Math\'ematiques
  (1966), no.~29, 95--103.

\bibitem{HR1984}
C.~W. Henson and L.~A. Rubel, \emph{Some applications of {{Nevanlinna}} theory
  to mathematical logic: Identities of exponential functions}, Transactions of
  the American Mathematical Society \textbf{282} (1984), no.~1, 1--32.

\bibitem{Lan1966}
S.~Lang, \emph{Introduction to transcendental numbers}, Addison-Wesley
  Publishing Co., Reading, Mass.-London-Don Mills, Ont., 1966. \MR{214547}

\bibitem{Lan1999}
\bysame, \emph{Complex analysis}, fourth ed., Graduate {{Texts}} in
  {{Mathematics}}, vol. 103, {Springer-Verlag, New York}, 1999.

\bibitem{Man2016}
V.~Mantova, \emph{Polynomial-exponential equations and {{Zilber}}'s
  conjecture}, Bulletin of the London Mathematical Society \textbf{48} (2016),
  no.~2, 309--320, with an appendix by V.\ Mantova and U.\ Zannier.

\bibitem{Mar2006}
D.~Marker, \emph{A remark on {{Zilber}}'s pseudoexponentiation}, The Journal of
  Symbolic Logic \textbf{71} (2006), no.~3, 791--798.

\bibitem{Mum1983}
D.~Mumford, \emph{Tata lectures on theta. {{I}}}, Progress in {{Mathematics}},
  vol.~28, {Birkh\"auser Boston, Inc., Boston, MA}, 1983.

\bibitem{Mum1984}
\bysame, \emph{Tata lectures on theta. {{II}}}, Progress in {{Mathematics}},
  vol.~43, {Birkh\"auser Boston, Inc., Boston, MA}, 1984.

\bibitem{Pel2013}
F.~Pellarin, \emph{On the generalized {{Carlitz}} module}, Journal of Number
  Theory \textbf{133} (2013), no.~5, 1663--1692.

\bibitem{Ser1988}
J-P. Serre, \emph{Algebraic groups and class fields}, Graduate {{Texts}} in
  {{Mathematics}}, vol. 117, {Springer-Verlag, New York}, 1988.

\bibitem{Dri1984}
L.~van~den Dries, \emph{Exponential rings, exponential polynomials and
  exponential functions}, Pacific Journal of Mathematics \textbf{113} (1984),
  no.~1, 51--66.

\bibitem{Wei1967}
K.~Weierstrass, \emph{Mathematische {{Werke}}. {{IV}}. {{Vorlesungen}} \"uber
  die {{Theorie}} der {{Abelschen Transcendenten}}}, {Georg Olms
  Verlagsbuchhandlung, Hildesheim; Johnson Reprint Corp., New York}, 1967
  (german).

\bibitem{ZS1958}
O.~Zariski and P.~Samuel, \emph{Commutative algebra, {V}olume {I}}, The
  University Series in Higher Mathematics, D. Van Nostrand Co., Inc.,
  Princeton, NJ, 1958, With the cooperation of I. S. Cohen. \MR{90581}

\bibitem{ZS1960}
\bysame, \emph{Commutative algebra. {V}ol. {II}}, The University Series in
  Higher Mathematics, D. Van Nostrand Co., Inc., Princeton,
  N.J.-Toronto-London-New York, 1960. \MR{120249}

\bibitem{Zil2002}
B.~Zilber, \emph{Exponential sums equations and the {{Schanuel}} conjecture},
  Journal of the London Mathematical Society. Second Series \textbf{65} (2002),
  no.~1, 27--44.

\bibitem{Zil2005}
\bysame, \emph{Pseudo-exponentiation on algebraically closed fields of
  characteristic zero}, Annals of Pure and Applied Logic \textbf{132} (2005),
  no.~1, 67--95.

\end{thebibliography}

\end{document}